\documentclass[review]{elsarticle}

\usepackage{hyperref}

\journal{Journal of Differential Equations}

\usepackage{geometry}
\usepackage{amssymb,amsmath,amssymb,mathrsfs}
\newtheorem{theorem}{Theorem}[section]
\newtheorem{proposition}{Proposition}[section]
\newtheorem{lemma}{Lemma}[section]

\newtheorem{remark}{Remark}[section]
\newproof{proof}{Proof}
\numberwithin{equation}{section}

\usepackage{color}
\definecolor{red}{rgb}{1,0,0} 
\definecolor{blue}{rgb}{0,0,1} 

\begin{document}
	
	\begin{frontmatter}
		
		\title{Energy estimates and hypocoercivity analysis for a multi-phase  Navier-Stokes-Vlasov-Fokker-Planck system with uncertainty}

		\author[mymainaddress,mysecondaryaddress,mythirdaddress]{Shi Jin\corref{mycorrespondingauthor}}
		\cortext[mycorrespondingauthor]{Corresponding author}
		\ead{shijin-m@sjtu.edu.cn}
		
		\address[mymainaddress]{School of Mathematical Sciences, Shanghai Jiao Tong University, Shanghai200240, China}
		\address[mysecondaryaddress]{Institute of Natural Sciences, Shanghai Jiao Tong University, Shanghai200240, China}
		\address[mythirdaddress]{Ministry of Education, Key Laboratory in Scientific and Engineering Computing, Shanghai Jiao Tong University, Shanghai200240, China}
		
		\author[mymainaddress]{Yiwen Lin}
		\ead{linyiwen@sjtu.edu.cn}

		\begin{abstract}
			This paper is concerned with a kineitc-fluid model with random initial inputs in the fine particle regime, which is a system coupling the incompressible Navier–Stokes equations and the Vlasov–Fokker–Planck equations that model dispersed particles of different sizes.
			A uniform regularity for random initial data near the global equilibrium is established in some suitable Sobolev spaces by using energy estimates,  and  we also prove the energy decays exponentially in time by hypocoercivity arguments, which means that the long time behavior of the solution is insensitive to the random perturbation  in the initial data.
			For the generalized polynomial chaos stochastic Galerkin method (gPC-sG) for the  model, with initial data near the global equilibrium and smooth enough in the physical and random spaces, we prove that the gPC-sG method has spectral accuracy, uniformly in time and the Knudsen number, and the error decays exponentially in time.
		\end{abstract}
		
		
	\end{frontmatter}
	
	
	\section{Introduction}
	
	The study of  kinetic-fluid models for a mixture of flows, in which the particles represent the dispersed phase evolving in a dense fluid, is motivated by applications such as the dispersion of smoke or dust \cite{Friedlander1977}, biomedical modeling of spray \cite{Baranger2005} and coupled models in combustion theory \cite{Williams1985}. 
	See \cite{Prosperetti2007} for more details on the modeling of such multi-phase flows.
	
	Mathematically, these problems can be described by
	partial differential equations, where the evolution of the particle distribution function is driven by a combination of particle transport, Stokes drag force exerted by the surrounding fluid on the particle and Brownian motion of particles.
	This leads to the Navier-Stokes-Vlasov-Fokker-Planck system, first proposed in \cite{GoudonJabin2004a, GoudonJabin2004b}. 
	In this paper, we are interested in models that describe a large number of particles, {\it with distinct but fixed sizes}, interacting with fluids.
	We will ignore the  influence of the external potential (gravity, electrostatic force, centrifugal force, etc.), coagulation and fragmentation that occurs between particles that induce change of  particle sizes.
	As usual in fluid dynamic, the dense fluid phase is a liquid or dense gas described by macroscopic quantities (such as mass density, velocity, and temperature) and is therefore modeled by the Euler or Navier-Stokes equations, depending on time and space variables.
	Particles (e.g. droplets, bubbles) dispersed in the fluid are described by distribution functions in phase space and are modeled by kinetic equations, depending on time, space  and (microscopic) particle velocity.
	Thus, the unknowns for different phases may not  depend on the same set of variables, and particles and fluid systems are  coupled through nonlinear forcing terms. Such a coupling and nonlinearities pose new  difficulties in mathematical analysis as well numerical computations than uncoupled problems.  
	
	There are two physically important regimes in these problems: one is the light particle regime \cite{GoudonJabin2004a}, 
	and the other is the fine particle regime \cite{GoudonJabin2004b}.
	For the light particle regime, 
	the velocity of fluid is small compared to the typical molecular velocity of particles and
	the particles are light and have little influence on the fluid.
	For the fine particle regime, 
	particle sizes are small compared to typical length scales and
	the density of the fluid and the particles are the same.
	Both  regimes have  much smaller relaxation time compared to the typical time scale.
	In this paper, we focus on the fine particle regime.
	
	The study of existence, uniqueness, and regularity problems depends on the nature of the coupling and the complexity of the equations used to describe the fluid. 
	For example, in \cite{Hamdache1998}, Hamdache established global existence and large-time behavior of solutions for the Vlasov–Stokes system. Boudin, Desvillettes, Grandmont, and Moussa \cite{BoudinDesvillettes2009} proved the global existence of weak solutions to the incompressible Vlasov–Navier–Stokes system on a periodic domain. Later, this result was extended to a bounded domain by Yu \cite{Yu2013}. Goudon, He, Moussa, and Zhang \cite{GoudonHe2010} established the global-in-time existence of classical solutions near the equilibrium for the incompressible Navier–Stokes–Vlasov–Fokker–Planck system and investigated regularity properties of the solutions as well as their long time behavior; 
	meanwhile Carrillo, Duan, and Moussa \cite{CarrilloDuan2011} studied the corresponding inviscid case. Chae, Kang, and Lee \cite{ChaeKang2011} obtained the global existence of weak and classical solutions for the Navier–Stokes–Vlasov–Fokker–Planck equations in a torus. Benjelloun, Desvillettes, and Moussa \cite{BenjellounDesvillettes2014} obtained the existence of global weak solutions to the incompressible Vlasov–Navier–Stokes system with a fragmentation kernel. Goudon, Jabin, and Vasseur \cite{GoudonJabin2004a,GoudonJabin2004b} investigated the hydrodynamic limits to the incompressible Vlasov–Navier–Stokes system under suitable scalings.
	Recently, Cao and Jiang \cite{CaoJiang2021} obtained a global bounded weak entropy solution for such one-dimensional Euler-Vlasov equations with arbitrarily large initial data.
	Numerical methods for  kinetic-fluid coupled systems that possess the asymptotic-preserving properties \cite{Jin1999} were developed 
	in \cite{Carrillo2008,Goudon2012,Goudon2013}.

	So far most of the aforementioned  references do not address the effect of size variations of particles, namely all particles have the same size.  This paper considers   kinetic-fluid model for a mixture of the flows for particles with {\it {distinct}} sizes.
	Such multi-size particle systems have a wide range of applications in engineering, especially for the complex meteorological simulation of large aircraft icing process. 
	For the distribution of droplets in the air, the impact of larger droplets contained in the droplet distribution cannot be ignored\cite{Cober2006,Hauf2006,Potapczuk2019}.
	However, it is very difficult to simulate 
	multi-size particles
	by experimental means. 
	Therefore, in this paper, we study mathematically the fluid-particles systems with distinct particle sizes and the uncertainty quantification aspect of the problem.

	Most of the work on kinetic-fluid  models are deterministic. However, there are many sources of uncertainty in these models. For example, initial  and boundary data are often obtained from experiments and therefore inevitably  have  measurement errors. Uncertainty can also come from the modeling of drag force, particle diffusion, etc. It is important to quantify these uncertainties because such quantification can help us understand how uncertainties affect the solution and thus make reliable predictions.
	Among popular  methods for uncertainty quantification (UQ) include the Monte Carlo (MC) method, the stochastic collocation (sC) method and the stochastic Galerkin (sG) method \cite{Xiu2009,Xiu2010,Gunzburger2014}. 
	
	We will consider our problems with uncertainty, characterized by random inputs in the initial data.  In order to analyze the accuracy of UQ methods, it is important to first analyze the regularity of the analytical solution in the random spaces. 
	For the sG method, regularity of truncated series approximations is required, but it is not straightforward to prove accuracy from regularity of the random space  due to the Galerkin projection error. 
	Instead, an evolution equation for the error can be derived and then estimates from the regularity in the random space of the exact solution can be achieved. 
	There have been a series of recent efforts to study the uniform (in the Knudsen number)  regularity in random space and/or local sensitivity for various types of kinetic equations, 
	including Jin et al. \cite{JinLiuMa2017} for linear transport equations, Jin-Zhu \cite{JinZhu2018} for the Vlasov-Poisson-Fokker-Planck equation, Jin-Liu \cite{JinLiu2017} and Liu \cite{Liu2018} for the linear semiconductor Boltzmann equation. In addition, the uniform regularity for the general linear transport equations conserving mass based on hypocoercivity is established in \cite{LiWang2017}. Uniform regularity is also obtained for nonlinear kinetic equations, such as Jin-Zhu \cite{JinZhu2018} for the Vlasov-Poisson-Fokker-Planck system, Shu-Jin \cite{ShuJin2018} for the Fokker-Planck-incompressible Navier-Stokes system based on \cite{GoudonHe2010}. A general framework for nonlinear collisional kinetic equations using hypocoercivity analysis in the random space  was provided in \cite{LiuJin2018}. 
	\cite{JinLiuMa2017, JinLiu2017, Liu2018} also prove the spectral accuracy for the sG method. 
	For   the linear convection-diffusion equation, the two velocity BGK model and the Fokker–Planck equation, \cite{ArnoldJin2020} gives an explicit construction of Lyapunov functionals that yields sharp decay estimates, including an extension to defective ODE systems. Numerically, the first UQ work for kinetic problems was done  in Jin, Xiu, and Zhu \cite{JinXiuZhu2015}, in which the notion of stochastic asymptotic-preserving (s-AP) was introduced. For later development see surveys \cite{HuJin2017, JinICM}.
	
	In this paper, we only consider the uncertainty from the initial data. To model the uncertainty here,let the velocity of the fluid and the distribution functions of particles with $N$ different sizes depend on the random variable $z$ (i.e., $u = u(t,x, z)$ and $F_i = F_i(t,x,v, z), i=1,2,\ldots,N$), which lives in the random space $\mathbb{Z}$ with probability distribution $\pi(z) dz$. The uncertainty from the initial data is then described by letting the initial data $u_0$ and $\{F_{i,0}\}_{i=1}^N$ depend on $z$. For notational clarity of our analysis, we  assume that the random space $\mathbb{Z}$ is one-dimensional. Our results can be extended to the case of multi-dimensional random spaces.

	Clearly, the convergence of numerical methods, such as the sG and sC methods, in the random space requires the regularity in the random space (hereafter be referred to $z$-regularity) \cite{JinLiuMa2017}. 
	Therefore, we first analyze the $z$-regularity for random initial data near the global equilibrium (where $F_i = \mu_i + \sqrt{\mu_i} f_i$ and $\mu_i$ is the local normalized Maxwellian) in some suitable Sobolev spaces (with derivatives with respect to $x$ and $z$). We use energy estimates and hypocoercivity arguments on the $z$-derivatives of $u$ and $f_i$. 
	The results show that for near-equilibrium initial data regularly dependent on $x$ and $z$, the solution always preserves the  regularity of the initial data and is insensitive to random perturbations of the initial data for large time. Then for the sG method, we consider the most popular choice of basis functions, the generalized polynomial chaos (gPC) \cite{Xiu2002}, i.e. orthogonal polynomials with respect to $\pi(z)dz$. We write equations for the gPC coefficients and do energy estimates, where we manage to make the estimate independent of the number $K$ of basis functions.  Finally, we write out the error equation for the gPC-sG method and do energy and hypocoercivity estimates. The results show that if the random initial data $(u_0, \{f_{i,0}\}_{i=1}^N)$ is sufficiently small in some suitable Sobolev spaces, the gPC-sG method has spectral accuracy, uniformly in time and the Knudsen number, and captures the exponential decay in time toward the global equilibrium of the exact solution. 
	
	One of the major difficulties in the hypocoercivity analysis of the gPC-SG method  is that a naive energy estimate for the gPC coefficients requires a small initial data condition depending on the number $K$ of basis functions, since the nonlinear term in the system produces a large number  ($K^3$) of terms in the equation  of gPC coefficients. But it is desirable to have a small initial data condition that is independent of the numerical parameter $K$, which means that the accuracy results are correct for this set of initial data for all $K$. To overcome this difficulty, we introduce a weighted sum of the  Sobolev norm of the gPC coefficients, which allows us to combine some terms together as part of a convergent series and use $K$-independent estimates to control nonlinear terms.
	
	The paper is organized as follows: 
	in Section 2, we describe precisely the PDE system of interest to us, including the hydrodynamic limit system and the system near the global equilibrium; 
	in Section 3, we state the energy estimate and hypocoercivity estimate for the $z$-derivatives of $u$ and $f_i$ and give the proof;
	in Section 4 we state and prove the spectral accuracy and long-time behavior of the sG method. 
	Finally, we conclude in Section 5.

	\section{Model Problem}
	
	In this paper, we focus on the fine particle regime, in which
	the suitably scaled PDE systems for the multi-phase model  are given by: 
	\begin{equation}\label{ModelEquation0}
		\left\{\begin{aligned}
			&\begin{aligned}
				(\tilde{F_i})_{t}+v \cdot \nabla_{x} (\tilde{F_i})&=\dfrac{1}{\epsilon}\dfrac{1}{i^{2/3}}\operatorname{div}_{v}\left((v-\tilde{u}) \tilde{F_i}+\dfrac{\bar{\theta}}{i}\nabla_{v} \tilde{F_i}\right),\\
				&\hspace{7em} \quad(t, x, v) \in \mathbb{R}^{+} \times \mathbb{T}^{3} \times \mathbb{R}^{3}, i=1,2,...,N, 
			\end{aligned}
			\\
			&\tilde{u}_{t}+(\tilde{u} \cdot \nabla_{x}) \tilde{u}+\nabla_{x} \tilde{p}-\Delta_{x} \tilde{u}=\dfrac{\kappa}{\epsilon} \sum_{i=1}^N\int_{\mathbb{R}^{3}}(v-\tilde{u}) \tilde{F_i} i^{1/3} \mbox{d} v, \quad(t, x) \in \mathbb{R}^{+} \times \mathbb{T}^{3}, \\
			&\nabla_{x} \cdot \tilde{u}=0,
		\end{aligned}\right.
	\end{equation}
	with the initial condition
	$$\left.\tilde{u}\right|_{t=0}=\tilde{u}_{0}, \quad \nabla_{x} \cdot \tilde{u}_{0}=0,\left.\quad \tilde{F_i}\right|_{t=0}=\tilde{F}_{i,0},$$
	where $t\geq 0$ is  time, $x\in \mathbb{T}= [-\pi,\pi]^3$ is the space variable, and $v\in \mathbb{R}^3$ is the particle velocity.
	For simplicity the  periodic boundary condition in the space domain is assumed. The fluid is described by its velocity field $\tilde{u}(t,x)\in \mathbb{R}^{3}$ and its pressure $\tilde{p}(t,x)$. The particles are described by their distribution function $\tilde{F_i} = \tilde{F_i}(t,x,v), i=1,2,\ldots,N$ in phase space. $N$ is the number of sizes of particles.
	$\epsilon$ is the Knudsen number, which satisfies $0<\epsilon\leq 1$. $\epsilon=O(1)$ corresponds to the kinetic regime, while $\epsilon\rightarrow 0$ corresponds to the fluid regime. $\kappa>0$ is the coupling constant, which equals the ratio between the particle density and fluid density.

	This system satisfies the following conservation properties. 
	\begin{itemize}
		\item mass conservation
		$$\frac{d}{d t} \sum_{i=1}^{N} \int_{\mathbb{T}^{3} \times \mathbb{R}^{3}}i \tilde{F}_i \,\mathrm{d} v\mathrm{d}x=0,$$
		\item momentum  conservation
		\begin{equation}\label{momentum}
			\frac{d}{d t}\left(\int_{\mathbb{T}^{3}} \tilde{u} \,\mathrm{d} x+\sum_{i=1}^{N} \int_{\mathbb{T}^{3} \times \mathbb{R}^{3}} i v \tilde{F}_i \,\mathrm{d} v\mathrm{d}x\right)=0,
		\end{equation}
		\item energy-entropy dissipation
		$$
		\begin{aligned}
			\frac{\mathrm{d}}{\mathrm{d} t}\left(\kappa\sum_{i=1}^{N} \int_{\mathbb{T}^{3}} \int_{\mathbb{R}^{3}} \tilde{F}_i\left(\ln (\tilde{F}_i)+1+i\frac{|v|^{2}}{2}\right) \,\mathrm{d} v \mathrm{d} x+\int_{\mathbb{T}^{3}}\frac{|\tilde{u}|^{2}}{2}\, \mathrm{d} x\right)+\int_{\mathbb{T}^{3}}\left|\nabla_{x} \tilde{u}\right|^{2} \,\mathrm{d} x\\
			+\frac{\kappa}{\epsilon}\sum_{i=1}^{N} \int_{\mathbb{T}^{3}} \int_{\mathbb{R}^{3}}i^{1/3}\left|(v-\tilde{u}) \sqrt{\tilde{F}_i}+\frac{\theta}{i} \frac{\nabla_{v} \tilde{F}_i}{\tilde{F}_i}\right|^{2} \, \mathrm{d} v \mathrm{d} x=0.
		\end{aligned}$$
	\end{itemize}

	\subsection{Hydrodynamic limit}

	For the deterministic multi-size particle-fluid systems \eqref{ModelEquation0}, we associate to $\tilde{F_i}(t,x,v), i=1,2,\ldots,N$ the following macroscopic quantities:
	$$
	\begin{gathered}
		n_{i}(t, x)=\int_{\mathbb{R}^{3}} \tilde{F_i}(t, x, v) \mathrm{d} v, \quad \rho_i(t,x) = i n_i(t,x), 
		\quad J_{i}(t, x)= i \int_{\mathbb{R}^{3}} v \tilde{F_i}(t, x, v) \mathrm{d} v, \\
		\mathbb{P}_{i}(t, x)=i \int_{\mathbb{R}^{3}} v \otimes v \tilde{F_i}(t, x, v) \mathrm{d} v ,
	\end{gathered}
	$$
	where $\rho_i, J_i$ and $\mathbb{P}_{i}$ are the mass, momentum and stress tensors, respectively, of particles of size $i$.
	Integrating the first equation in \eqref{ModelEquation0} with respect to $i \mathrm{~d} v$ and $i v \mathrm{~d} v$ respectively, one obtains
	$$
	i \partial_{t} n_{i}+\nabla_{x} \cdot J_{i}=0,
	$$
	and
	$$\partial_{t} J_{i}+\operatorname{Div}_{x} \mathbb{P}_{i}=-\frac{1}{\epsilon} \frac{1}{i^{2 / 3}}\left(J_{i}-i n_{i} \tilde{u}\right).$$
	We remark that  system \eqref{ModelEquation0} conserves the total momentum since
	\begin{equation}\label{eq:kappaJ}
		\partial_{t}\left(\tilde{u}+ \kappa \sum_{i=1}^{N} J_{i}\right)+\operatorname{Div}_{x}\left(\tilde{u} \otimes \tilde{u}+\kappa \sum_{i=1}^{N} \mathbb{P}_{i}\right)+ \nabla_{x} \tilde{p}-\Delta_{x} \tilde{u}=0.
	\end{equation}
	Accordingly, for $\epsilon<<1$,  $J_i$ and $\mathbb{P}_i$ are approximated by the moments of the Maxwellian, i.e., 
	$$J_i \simeq i n_i \tilde{u}, \quad \mathbb{P}_i \simeq i n_i \tilde{u} \otimes \tilde{u}+ i n_i  \mathbb{I}.$$
	Inserting this ansatz into \eqref{eq:kappaJ}, one arrives at
	\begin{equation}
		\partial_{t}\left(1\left(1+ \kappa \sum_{i=1}^{N}  i n_i\right)\tilde{u}\right)+\operatorname{Div}_{x}\left(\left(1+ \kappa \sum_{i=1}^{N}  i n_i\right)\tilde{u} \otimes \tilde{u}\right)+ \nabla_{x} \left(\tilde{p}+ \kappa \sum_{i=1}^{N}  i n_i\right)-\Delta_{x} \tilde{u}=0, 
	\end{equation}
	
	Thus, as $\epsilon\rightarrow 0$,  \eqref{ModelEquation0} has a hydrodynamic limit
	\begin{equation}\label{eq:iniu}
		\left\{\begin{aligned}
			& \partial_{t} n_i +\nabla_{x} \cdot(n_i \tilde{u})=0, \\
			&\partial_{t}\left(\left(1+ \kappa \sum_{i=1}^{N}  i n_i\right)\tilde{u}\right)+\operatorname{Div}_{x}\left(\left(1+ \kappa \sum_{i=1}^{N}  i n_i\right)\tilde{u} \otimes \tilde{u}\right)+ \nabla_{x} \left(\tilde{p}+ \kappa \sum_{i=1}^{N}  i n_i\right)-\Delta_{x} \tilde{u}=0, \\
			&\nabla_{x} \cdot \tilde{u}=0.
		\end{aligned}\right.
	\end{equation}
	Denote $\nu = \displaystyle\sum_{i=1}^{N}  i n_i$.
	Then system \eqref{eq:iniu} becomes
	\begin{equation}\label{eq:nu}
		\left\{\begin{aligned}
			& \partial_{t} \nu +\nabla_{x} \cdot(\nu \tilde{u})=0, \\
			&\partial_{t}\left(\left(1+ \kappa \nu\right)\tilde{u}\right)+\operatorname{Div}_{x}\left(\left(1+ \kappa \nu\right)\tilde{u} \otimes \tilde{u}\right)+ \nabla_{x} \left(\tilde{p}+ \kappa \nu \right)-\Delta_{x} \tilde{u}=0, \\
			&\nabla_{x} \cdot \tilde{u}=0,
		\end{aligned}\right.
	\end{equation}
	which is the incompressible Navier-Stokes system for the composite and inhomogeneous density $(1+\kappa \nu)$.
	
	\subsection{near equilibrium}

	Consider the local normalized Maxwellian
	$$\mu_i(v)=\frac{1}{(\frac{2 \pi \bar{\theta}}{i})^{3 / 2}|\mathbb{T}|^{3}} e^{-\frac{i v^{2}}{2\bar{\theta}} },$$
	and look at solutions of \eqref{ModelEquation0} which read
	\begin{equation}\label{eq:F_i0}
		\tilde{F_i}= \mu_i + \sqrt{\mu_i} \tilde{f}_i.
	\end{equation}
	Plugging \eqref{eq:F_i0} into \eqref{ModelEquation0}, one obtains the following system for the perturbation $(\tilde{u},\{\tilde{f}_i\}_{i=1}^N)$:
	\begin{equation}\label{eq:uf0}
		\left\{\begin{aligned}
			&(\tilde{f}_i)_{t}+v \cdot \nabla_{x} (\tilde{f}_i)+\frac{1}{i^{2/3}\epsilon}\tilde{u} \cdot\left(\nabla_{v} \tilde{f}_i-\frac{ i v}{2 \bar{\theta}} \tilde{f}_i\right)-\frac{i^{1/3}}{\bar{\theta}\epsilon}\tilde{u} \cdot v \sqrt{\mu_i}=\frac{1}{i^{2/3}\epsilon}\left(-\frac{i}{\bar{\theta}}\frac{|v|^{2}}{4} \tilde{f}_i+\frac{3}{2} \tilde{f}_i+\frac{\bar{\theta}}{i}\Delta_{v} \tilde{f}_i\right), \\
			&\tilde{u}_{t}+\tilde{u} \cdot \nabla_{x} \tilde{u}+\nabla_{x} \tilde{p}-\Delta_{x} \tilde{u}+\frac{\kappa}{\epsilon}\tilde{u}\sum_{i=1}^{N} i^{1/3} +\frac{\kappa}{\epsilon} \tilde{u} \sum_{i=1}^{N}i^{1/3}  \int_{\mathbb{R}^{3}} \sqrt{\mu_i} \tilde{f}_i  \mathrm{d} v-\frac{\kappa}{\epsilon}\sum_{i=1}^{N} i^{1/3}  \int_{\mathbb{R}^{3}} v \sqrt{\mu_i} \tilde{f}_i \mathrm{d} v=0, \\
			&\nabla_{x} \cdot \tilde{u}=0,
		\end{aligned}	
		\right.
	\end{equation}
	with the initial data
	\begin{equation*}
		\left.\tilde{u}\right|_{t=0}=\tilde{u}_{0},\left.\quad \tilde{f}_i\right|_{t=0}=\tilde{f}_{i,0}, \quad \  \int_{\mathbb{T}^{3}}\int_{\mathbb{R}^{3}} \sqrt{\mu_i} \tilde{f}_{i,0} \,\mathrm{d} v \mathrm{d} x=0,
	\end{equation*}
	and
	\begin{equation*}
		\int_{\mathbb{T}^{3}} \tilde{u}_{0} \mathrm{d} x+\sum_{i=1}^{N}\int_{\mathbb{T}^{3}} \int_{\mathbb{R}^{3}}i v \sqrt{\mu_i} f_{i,0} \,\mathrm{d} v \mathrm{d} x=0, \quad \quad \nabla_x \cdot \tilde{u}_{0}=0,
	\end{equation*}
	which show that the fluctuations of the initial data, $(\tilde{u}_0, \{\tilde{f}_{i,0}\}_{i=1}^N)$, do not affect the total momentum and mass, and the perturbation of the fluid velocity is divergence free.
	
	Rigorous proofs of existence-uniqueness for this kinetic-fluid multi-phase flow model \eqref{eq:uf0} with distinct particle sizes near equilibrium will be investigated in future work. This paper focuses on the energy estimate and hypocoercivity analysis for such multi-phase systems with uncertainties.

	\section{Uncertainty and Regularity in the random space}
	
	To model the uncertainty, let the velocity field of the fluid and the distribution functions of  particles depend on the random variable $z$ (i.e., $u = u(t,x, z)$ and $F_i = F_i(t,x,v, z), i=1,2,\ldots,N$), which lives in the random space $\mathbb{Z}$ with probability distribution $\pi(z) dz$. 
	For clarity of notations, we  assume that the random space $\mathbb{Z}$ is one-dimensional. The PDE systems for a mixture of flows with uncertainties are given by: 
	\begin{equation}\label{ModelEquation}
		\left\{\begin{aligned}
			&\begin{aligned}
				(F_i)_{t}+v \cdot \nabla_{x} (F_i)&=\dfrac{1}{\epsilon}\dfrac{1}{i^{2/3}}\operatorname{div}_{v}\left((v-u) F_i+\dfrac{\bar{\theta}}{i}\nabla_{v} F_i\right), \\
				&\hspace{6em} (t, x, v, z) \in \mathbb{R}^{+} \times \mathbb{T}^{3} \times \mathbb{R}^{3} \times \mathbb{Z}, i=1,2,...,N, 
			\end{aligned}\\
			&u_{t}+(u \cdot \nabla_{x}) u+\nabla_{x} p-\Delta_{x} u=\dfrac{\kappa}{\epsilon} \sum_{i=1}^N\int_{\mathbb{R}^{3}}(v-u) F_i i^{1/3} \mbox{d} v, \quad(t, x,z) \in \mathbb{R}^{+} \times \mathbb{T}^{3}, \\
			&\nabla_{x} \cdot u=0,
		\end{aligned}\right.
	\end{equation}	 
	with the initial condition that depends on $z$:
	$$\left.u\right|_{t=0}=u_{0}, \quad \nabla_{x} \cdot u_{0}=0,\left.\quad F_i\right|_{t=0}=F_{i,0},$$
	where $u_{0}=u(0,x,z)$, $F_{i,0}=F_i(0,x,v,z), i=1,2,\ldots,N$.
	
	Denote by $F_i = \mu_i + \sqrt{\mu_i} f_i$ as in \eqref{eq:F_i0}. 
	The following system for the perturbation $(u,\{f_i\}_{i=1}^N)$ is achieved: 
	\begin{equation}\label{eq:uf}
		\left\{\begin{aligned}
			&(f_i)_{t}+v \cdot \nabla_{x} (f_i)+\frac{1}{i^{2/3}\epsilon}u \cdot\left(\nabla_{v} \frac{ i v}{2 \bar{\theta}} \right)f_i-\frac{i^{1/3}}{\bar{\theta}\epsilon}u \cdot v \sqrt{\mu_i}=\frac{1}{i^{2/3}\epsilon}\left(-\frac{i}{\bar{\theta}}\frac{|v|^{2}}{4}+\frac{3}{2} +\frac{\bar{\theta}}{i}\Delta_{v} \right)f_i, \\
			&u_{t}+u \cdot \nabla_{x} u+\nabla_{x} p-\Delta_{x} u+\frac{\kappa}{\epsilon}u\sum_{i=1}^{N} i^{1/3} +\frac{\kappa}{\epsilon} u \sum_{i=1}^{N}i^{1/3}  \int_{\mathbb{R}^{3}} \sqrt{\mu_i} f_i  \mathrm{d} v-\frac{\kappa}{\epsilon}\sum_{i=1}^{N} i^{1/3}  \int_{\mathbb{R}^{3}} v \sqrt{\mu_i} f_i \mathrm{d} v=0, \\
			&\nabla_{x} \cdot u=0,
		\end{aligned}	
		\right.
	\end{equation}
	with the initial data
	\begin{equation}
		\left.u\right|_{t=0}=u_{0},\left.\quad f_i\right|_{t=0}=f_{i,0},
	\end{equation}
	\begin{equation}\label{initialu0}
		\int_{\mathbb{T}^{3}} u_{0} \mathrm{d} x+\sum_{i=1}^{N}\int_{\mathbb{T}^{3}} \int_{\mathbb{R}^{3}}i v \sqrt{\mu_i} f_{i,0} \,\mathrm{d} v \mathrm{d} x=0, \quad \quad \nabla_x \cdot u_{0}=0,
	\end{equation}
	and
	\begin{equation}\label{initialu0_2}
		\int_{\mathbb{T}^{3}}\int_{\mathbb{R}^{3}} \sqrt{\mu_i} f_{i,0} \,\mathrm{d} v \mathrm{d} x=0.
	\end{equation}
	
	Define the mean fluid velocity
	$$
	\bar{u}(t, z) \stackrel{\text { def }}{=} \frac{1}{|\mathbb{T}|^{3}} \int_{\mathbb{T}^{3}} u(t, x, z) \mathrm{d} x .
	$$
	Averaging the second equation in \eqref{eq:uf} yields 
	\begin{equation}\label{eq:ubar}
		\bar{u}_{t} +\frac{\kappa}{\epsilon} \sum_{i=1}^{N} i^{1/3} \bar{u}+\frac{\kappa}{\epsilon}\frac{1}{|\mathbb{T}|^{3}}\sum_{i=1}^{N} \int_{\mathbb{T}^{3}}  \int_{\mathbb{R}^{3}} \sqrt{\mu_i} u f_i i^{1/3}\mathrm{d} v \mathrm{d} x = \frac{\kappa}{\epsilon}\frac{1}{|\mathbb{T}|^{3}}\sum_{i=1}^{N} i^{1/3}  \int_{{\mathbb{T}}^{3}}\int_{\mathbb{R}^{3}} v \sqrt{\mu_i} f_i \mathrm{d} v \mathrm{d} x .
	\end{equation}
	On the other hand, the momentum conservation \eqref{momentum} together with the condition \eqref{initialu0} implies
	\begin{equation}\label{eq:ubar2}
		-\frac{1}{|\mathbb{T}|^{3}}\sum_{i=1}^{N}\int_{{\mathbb{T}}^{3}}\int_{{\mathbb{R}}^{3}}i v \sqrt{\mu_i} f_i \mathrm{d} v \mathrm{d} x=\frac{1}{|\mathbb{T}|^{3}}\int_{\mathbb{T}^{3}} u \mathrm{d} x=\bar{u} .
	\end{equation}
	By multiplying the above by $\bar{u}$, one has
	\begin{equation}\label{ineq:ubar}
		\frac{1}{2} \frac{d}{d t}|\bar{u}|^{2}+\frac{\kappa}{\epsilon} \sum_{i=1}^{N} i^{1/3}|\bar{u}|^{2}
		+\frac{\kappa}{\epsilon}\frac{\bar{u}}{|\mathbb{T}|^{3}} \sum_{i=1}^{N} \int_{\mathbb{T}^{3}}  \int_{\mathbb{R}^{3}} \sqrt{\mu_i} u f_i i^{1/3}\mathrm{d} v \mathrm{d} x \leq \frac{\kappa}{\epsilon} |\bar{u}|^{2}.
	\end{equation}

	\subsection{Notations}
	
	Due to the random variable $z$, our notation is different from \cite{GoudonHe2010} and similar to \cite{ShuJin2018} with different definitions of inner products related to hypocoercivity arguments. All the norms or inner products with a single bound (like $|\cdot|,(\cdot, \cdot),[\cdot, \cdot])$, integral in $x, v$ and pointwise in $z$, is a function of $z$. All the norms or inner products with a double bound (like $\|\cdot\|,((\cdot, \cdot)),[[\cdot, \cdot]])$, integral with respect to all variables, is a number.
	
	Let $\alpha=\left(\alpha_{1}, \alpha_{2}, \alpha_{3}\right)$ be a multi-index. Then define
	$
	\partial^{\alpha}=\partial_{x_{1}}^{\alpha_{1}} \partial_{x_{2}}^{\alpha_{2}} \partial_{x_{3}}^{\alpha_{3}}.
	$
	The $z$-derivative of order $\gamma$ of a function $f$ is denoted by
	$
	f^{\gamma}=\partial_{z}^{\gamma} f.
	$
	
	For function $u=u(x), f=f(x, v)$, define the Sobolev norm (with $x$-derivatives)
	$$
	\|u\|_{s}^{2}=\sum_{|\alpha| \leq s}\left\|\partial^{\alpha} u\right\|_{L_{x}^{2}}^{2}, \quad\|f\|_{s}^{2}=\sum_{|\alpha| \leq s}\left\|\partial^{\alpha} f\right\|_{L_{x, v}^{2}}^{2}.
	$$
	In particular, denote by $\|u\|_{0}$  the $L_{x}^{2}$ norm of $u$. 
	
	For function $u=u(x, z), f=f(x, v, z)$, define the sum of Sobolev norms
	\begin{equation}\label{Sobolevnorm}
		|u|_{s, r}^{2}=\sum_{|\gamma| \leq r}\left\|u^{\gamma}(\cdot, z)\right\|_{s}^{2}, \quad|f|_{s, r}^{2}=\sum_{|\gamma| \leq r}\left\|f^{\gamma}(\cdot, \cdot, z)\right\|_{s}^{2},	
	\end{equation}
	where $|u|_{s, r}$ and $|f|_{s, r}$ are functions of $z$. Then define the expected value of the total Sobolev norm by
	$$
	\|u\|_{s, r}^{2}=\int|u|_{s, r}^{2} \pi(z) \mathrm{d} z, \quad\|f\|_{s, r}^{2}=\int|f|_{s, r}^{2} \pi(z) \mathrm{d} z.
	$$
	
	For function $\bar{u}=\bar{u}(z)$, also define the sum of derivatives and the Sobolev norm by
	$$
	|\bar{u}|_{r}^{2}=\sum_{|\gamma| \leq r}\left|\bar{u}^{\gamma}\right|^{2}, \quad\|\bar{u}\|_{r}^{2}=\int|\bar{u}|_{r}^{2} \pi(z) \mathrm{d} z.
	$$
	In all these notations, the sub-index $r$ is omitted when $r=0$.
	
	The $L^{2}$ inner product of functions defined on $x$-space of $x, v$-space is denoted by $\langle\cdot, \cdot\rangle$, i.e.,
	$$
	\langle f, g\rangle=\int f g \mathrm{d} x, \quad \text { or } \quad\langle f, g\rangle=\iint f g \mathrm{d} v \mathrm{d} x .
	$$
	In case the inputs also depend on $z$, $\langle f, g\rangle$ only integrates in $x$ or $(x, v)$, and the inner product is a function of $z$. For example,
	$$
	\langle f, g\rangle(z)=\int f(x, z) g(x, z) \mathrm{d} x.
	$$

	Next we introduce the inner products related to the hypocoercivity arguments. Define
	$$
	\mathcal{K}_i=\dfrac{\bar{\theta}}{i}\nabla_{v}+\frac{v}{2}, \quad \mathcal{P}_i=\dfrac{\bar{\theta}}{i} v \cdot \nabla_{x}, \quad \mathcal{S}_{j}=\left[\mathcal{K}_{ij}, \mathcal{P}_i\right]=\mathcal{K}_{ij} \mathcal{P}_i-\mathcal{P}_i \mathcal{K}_{ij}=\dfrac{\bar{\theta}^2}{i^2}\partial_{x_{j}}, \quad \mathcal{K}_i^{*}=-\dfrac{\bar{\theta}}{i}\nabla_{v}+\frac{v}{2},
	$$
	where $\mathcal{K}_i^{*}$ is the adjoint operator of $\mathcal{K}_i$, in the sense that $\langle\mathcal{K}_i f_i, g_i\rangle=\left\langle f_i, \mathcal{K}_i^{*} \cdot g_i\right\rangle$, where $f_i$ has one component and $g_i$ has three components.

	For functions $f_i=f_i(x, v), g_i=g_i(x, v)$, define
	$$
	\begin{aligned}
		&(f_i, g_i)=\frac{1}{i^{1/3}}\left(2\langle\mathcal{K}_i f_i, \mathcal{K}_i g\rangle+\epsilon^2\langle\mathcal{K}_i f_i, \mathcal{S}_i g_i\rangle+\epsilon^2\langle\mathcal{S}_i f_i, \mathcal{K}_i g_i\rangle+\epsilon^3\langle\mathcal{S}_i f_i, \mathcal{S}_i g_i\rangle\right), \\
		&[f_i, g_i]=\langle\mathcal{K}_i f_i, \mathcal{K}_i g_i\rangle+\epsilon^4\langle\mathcal{S}_i f_i, \mathcal{S}_i g_i\rangle+\epsilon^2\left\langle\mathcal{K}_i^{2} f_i, \mathcal{K}_i^{2} g_i\right\rangle+\epsilon^4\langle\mathcal{K}_i \mathcal{S}_i f_i, \mathcal{K}_i \mathcal{S}_i g_i\rangle,
	\end{aligned}
	$$
	where we denote $\langle\mathcal{K}_i \mathcal{S}_i f_i, \mathcal{K}_i \mathcal{S}_i g_i\rangle:=\sum_{j, l=1}^{3}\left\langle\mathcal{K}_{ij} \mathcal{S}_{il} f_i, \mathcal{K}_{ij} \mathcal{S}_{il} g_i\right\rangle .$
	
	For functions $f=f(x, v, z), g=g(x, v, z)$, define
	\begin{equation}\label{def:fg}
		\begin{aligned}
			(f, g)_{s, r}=\sum_{|\gamma| \leq r} \sum_{|\alpha| \leq s}\left(\partial^{\alpha} f^{\gamma}(\cdot, \cdot, z), \partial^{\alpha} g^{\gamma}(\cdot, \cdot, z)\right),\\
			[f, g]_{s, r}=\sum_{|\gamma| \leq r} \sum_{|\alpha| \leq s}\left[\partial^{\alpha} f^{\gamma}(\cdot, \cdot, z), \partial^{\alpha} g^{\gamma}(\cdot, \cdot, z)\right],
		\end{aligned}
	\end{equation}
	where $(f, g)_{s, r}$ or $[f, g]_{s, r}$ is a function of $z$.
	
	Then we introduce the inner product in the $(x, v, z)$ space:
	$$
	\langle\langle f, g\rangle\rangle=\int\langle f, g\rangle \pi(z) \mathrm{d} z,\quad
	(( f, g))=\int( f, g)\pi(z) \mathrm{d} z,\quad
	[[ f, g]]=\int[ f, g] \pi(z) \mathrm{d} z,
	$$
	$$
	(( f, g))_{s, r}=\int( f, g)_{s, r}\pi(z) \mathrm{d} z,\quad
	[[ f, g]]_{s, r}=\int[ f, g]_{s, r} \pi(z) \mathrm{d} z.
	$$
	We also define the following norms in the $(x, v, z)$ space:
	\begin{equation}\label{def:finfty}
		\|u\|_{W^{s, \infty}} =\max _{|\alpha| \leq s}\left\|\partial^{\alpha} u\right\|_{L_{x, z}^{\infty}}, 
		\|f\|_{W^{s, \infty}} =\max _{|\alpha| \leq s}\left\|\partial^{\alpha} f\right\|_{L_{x, z}^{\infty}\left(L_{v}^{2}\right)}.
	\end{equation}
	
	\subsection{Main Results}
	
	Now we focus on the system \eqref{eq:uf} with the random variable $z$.
	
	Our first result is the following energy estimate assuming near-equilibrium initial data:
	
	\begin{theorem}\label{thm:energyestimate}
		Assume $(u, \{f_i\}_{i=1}^N)$ solves \eqref{eq:uf} with initial data verifying \eqref{initialu0}. Fix a point $z$. Define the energy
		\begin{equation}\label{def:E}
			E(t ; z)=E_{s, r}(t ; z)=|u|_{s,r}^{2}+\kappa\bar{\theta}\sum_{i=1}^{N}|f_i|_{s, r}^{2}+|\bar{u}|_{r}^{2},
		\end{equation}
		with integers $s \geq 2$ and $r \geq 0$. Then there exists a constant $c_{1}=c_{1}(s, r)>0$, such that $E(0 ; z) \leq c_{1}$ implies that $E(t ; z)$ is non-increasing in $t$.
	\end{theorem}
	
	This theorem is  proved in subsection \ref{sec:3.3} by an energy estimate on $\partial^{\alpha} f_i^{\gamma}$.
	
	From now on we omit the dependence on $z$ of $E$, in case there is no confusion.
	
	Next, by a standard hypocoercivity argument, we strengthen the above theorem into the following one:
	
	\begin{theorem}\label{thm:hypocoercivity}
		Assume $(u, \{f_i\}_{i=1}^N)$ solves \eqref{eq:uf} with initial data verifying \eqref{initialu0} and \eqref{initialu0_2}. There exists a constant $c_{1}^{\prime}(s, r)$ such that, if we assume $s \geq 0, E_{s+3, r}(0) \leq c_{1}^{\prime}(s, r)$, and that $C_{s, r}^{h}=\kappa\bar{\theta}\sum_{i=1}^N\left.(f_i, f_i)_{s, r}\right|_{t=0}$ (defined by \eqref{def:fg})  is finite, then there exists a constant $\lambda>0$ such that
		$$
		E_{s, r}(t) \leq C\left(E_{s, r}(0)+C_{s, r}^{h}\right) e^{-\lambda t},
		$$
		where $C=C(s, r)$.
	\end{theorem}
	
	Note that the constant $\lambda$ is independent of $\epsilon$ and depends on $\kappa$, $\bar{\theta}$ and $N$ (the number of particle sizes). All the constants $C$ in this paper are  independent of $\epsilon$ and may depend on $\kappa$, $\bar{\theta}$ and $N$ (the number of particle sizes). 
	
	This theorem implies that as long as the random perturbation $\left(u_{0}, \{f_{i,0}\}_{i=1}^N\right)$ on the initial data is small in suitable Sobolev spaces and has vanishing total mass and momentum, the long-time behavior of the solution is insensitive to the random initial data.

	\subsection{Energy estimate (Proof of Theorem \ref{thm:energyestimate})}\label{sec:3.3}
	
	\subsubsection{Preliminary estimates}
	
	We first state some lemmas on nonlinear estimates. Denote the space of functions with finite $\|\cdot\|_{s}$ norm as
	\begin{equation}\label{def:Hs}
		H^{s}=\left\{u(x):\|u\|_{s}<\infty\right\}, \quad \tilde{H}^{s}=\left\{f(x, v):\|f\|_{s}<\infty\right\} .
	\end{equation}
	The following three lemmas are from \cite{GoudonHe2010} and \cite{ShuJin2018}:

	\begin{lemma}\label{Lem1.1}
		Let $u=u(x) \in H^{s}, w=w(x) \in H^{s}, f=f(x, v) \in \tilde{H}^{s} .$ Then 
		$$
		\begin{aligned}
			&\|u w\|_{s} \leq C\|u\|_{s}\|w\|_{s}, \\
			&\|u f\|_{s} \leq C\|u\|_{s}\|f\|_{s},
		\end{aligned}
		$$
		for $s>3 / 2$,
		where $C=C(s)$.
	\end{lemma}
	
	\begin{lemma}\label{Lem1.2}
		Let $u=u(x, z) \in W_{z}^{r, \infty}\left(H^{s}\right), w=w(x, z) \in W_{z}^{r, \infty}\left(H^{s}\right), f=f(x, v, z) \in W_{z}^{r, \infty}(\tilde{H}^{s})$.  Let $|\gamma| \leq r .$ Then for $s>3 / 2$ and all $z$,
		$$\begin{aligned}
			&\left|(u w)^{\gamma}\right|_{s} \leq C|u|_{s, r}|w|_{s, r}, \\
			&\left|(u f)^{\gamma}\right|_{s} \leq C|u|_{s, r}|f|_{s, r},
		\end{aligned}$$
		where $C=C(s, r)$.
	\end{lemma}
	
	\begin{lemma}\label{Lem1.3}
		Let $u=u(x, z) \in W_{z}^{r, \infty}\left(H^{s}\right), w=w(x, z) \in W_{z}^{r, \infty}\left(H^{s}\right), y=y(x, z) \in W_{z}^{r, \infty}\left(H^{s}\right), f=f(x, v, z) \in$ $W_{z}^{r, \infty}(\tilde{H}^{s}), g=g(x, v, z) \in W_{z}^{r, \infty}(\tilde{H}^{s}) .$ Let $|\gamma| \leq r,|\alpha| \leq s .$ Then for $s>3 / 2$ and all $z$,
		$$\begin{aligned}
			\left|\left\langle\partial^{\alpha}(u w)^{\gamma}, y^{\gamma}\right\rangle\right| & \leq \frac{C(s, r)}{\delta}|u|_{s, r}^{2}|w|_{s, r}^{2}+\delta|y|_{0, r}^{2}, \\
			\left|\left\langle\partial^{\alpha}(u f)^{\gamma}, g^{\gamma}\right\rangle\right| & \leq \frac{C(s, r)}{\delta}|u|_{s, r}^{2}|f|_{s, r}^{2}+\delta|g|_{0, r}^{2},
		\end{aligned}$$
		where $\delta$ is any positive number.
	\end{lemma}
	
	\subsubsection{Proof of Theorem \ref{thm:energyestimate}}
	\begin{proof}
		Taking $z$-derivative of order $\gamma$ and $x$-derivative of order $\alpha$ of \eqref{eq:uf}, and taking $z$-derivative of order $\gamma$ of \eqref{eq:ubar}-\eqref{eq:ubar2} gives
		\begin{equation}\label{eq:ufgamma}
			\begin{aligned}
				&\partial_{t} \partial^{\alpha} f_i^{\gamma}+ v \cdot \nabla_{x} \partial^{\alpha} f_i^{\gamma}+\frac{i^{1/3}}{\bar{\theta}\epsilon}\left(\dfrac{\bar{\theta}}{i}\nabla_{v}-\frac{v}{2}\right) \cdot \partial^{\alpha}(u f_i)^{\gamma}
				\underbrace{-\frac{i^{1/3}}{\bar{\theta}\epsilon} \partial^{\alpha} u^{\gamma} \cdot v \sqrt{\mu_i}}\\
				&\begin{aligned}
					=\underbrace{\frac{i^{1/3}}{\bar{\theta}\epsilon}\left(\frac{-|v|^{2}}{4}+\frac{3}{2}\dfrac{\bar{\theta}}{i}+\dfrac{\bar{\theta}^2}{i^2}\Delta_{v}\right) \partial^{\alpha} f_i^{\gamma}},\\
					\partial_{t} \partial^{\alpha} u^{\gamma}+\partial^{\alpha}\left(u \cdot \nabla_{x} u\right)^{\gamma}+\nabla_{x} \partial^{\alpha} p^{\gamma} \underline{-\Delta_{x} \partial^{\alpha} u^{\gamma}} \underbrace{+\dfrac{\kappa}{\epsilon}\sum_{i=1}^{N}i^{1/3}\partial^{\alpha} u^{\gamma}}+\dfrac{\kappa}{\epsilon}\sum_{i=1}^{N}i^{1/3}\int \sqrt{\mu_i} \partial^{\alpha}(u f_i)^{\gamma} \mathrm{d} v\\ \underbrace{-\dfrac{\kappa}{\epsilon}\sum_{i=1}^{N}i^{1/3} \int v \sqrt{\mu_i} \partial^{\alpha} f_i^{\gamma} \mathrm{d} v}=0, \end{aligned}\\
				&\nabla_{x} \cdot \partial^{\alpha} u^{\gamma}=0, \\
				& \partial_{t} \bar{u}^{\gamma} \underline{+\dfrac{\kappa}{\epsilon}\sum_{i=1}^{N}i^{1/3} \bar{u}^{\gamma}}+\dfrac{\kappa}{\epsilon}\frac{1}{|\mathbb{T}|^{3}}\sum_{i=1}^{N}i^{1/3} \iint \sqrt{\mu_i}(u f_i)^{\gamma} \mathrm{d} v \mathrm{d} x =\underline{\frac{\kappa}{\epsilon} \frac{1}{|\mathbb{T}|^{3}} \sum_{i=1}^{N} \int_{{\mathbb{T}}^{3}}\int_{{\mathbb{R}}^{3}}  i^{1/3} v \sqrt{\mu_i} f_i^\gamma \mathrm{d} v \mathrm{d} x},\\			
				& -\frac{1}{|\mathbb{T}|^{3}}\sum_{i=1}^{N}\int_{{\mathbb{T}}^{3}}\int_{{\mathbb{R}}^{3}}i v \sqrt{\mu_i} f_i^\gamma \mathrm{d} v \mathrm{d} x=\bar{u}^\gamma .
			\end{aligned}
		\end{equation}
		Now do $L^{2}$ estimate on each equation above (except the third one), i.e., multiply the first equation by $\kappa\bar{\theta}\partial^{\alpha} f_i^{\gamma}$, integrate in $(v, x)$ and sum over $i$; multiply the second equation by $\partial^{\alpha} u^{\gamma}$ and integrate in $x$;  multiply the fourth equation by $\bar{u}^{\gamma}$. Finally add the results together and sum over $|\gamma| \leq r,|\alpha| \leq s$. Then one gets the following equation (at each $z$):
		$$
		\frac{1}{2} \partial_{t} E+G+B\leq 0,
		$$
		where the energy $E$ is defined in \eqref{def:E}. The good terms $G$ are given by
		$$
		G=\underline{G_{1}}+\underbrace{G_{2}}_{|\gamma| \leq s}=\underline{G_{1}}+\dfrac{\kappa}{\epsilon}\sum_{i=1}^{N}i^{1/3}\underbrace{G_{2,i}}_{|\gamma| \leq s}= G_{1, \gamma}+\dfrac{\kappa}{\epsilon}\sum_{i=1}^{N}i^{1/3}\sum_{|\gamma| \leq s} G_{2, i,\gamma},
		$$
		with
		$$
		\begin{aligned}
			G_{1, \gamma} &=\left|\nabla_{x} u^{\gamma}\right|_{s}^{2}+\dfrac{\kappa}{\epsilon}\left(\sum_{i=1}^{N}i^{1/3}-1\right)\left|\bar{u}^{\gamma}\right|^{2}, \\
			G_{2,i, \gamma} &=\left|u^{\gamma} \sqrt{\mu_i}-\frac{\bar{\theta}}{i} \nabla_{v} f_i^{\gamma}- \frac{v}{2} f_i^{\gamma}\right|_{s}^{2},
		\end{aligned}
		$$
		$G_{1}$ and $G_{2}$ come from the underlined terms and the underbraced terms in \eqref{eq:ufgamma} respectively. To verify the $G_{2}$ term, we provide the following calculation:
		$$
		\begin{aligned}
			\left\langle\partial^{\alpha} u^{\gamma}, \partial^{\alpha} u^{\gamma}\right\rangle-&\left\langle - \int v \sqrt{\mu_i} \partial^{\alpha} f_i^{\gamma} \mathrm{d} v, \partial^{\alpha} u^{\gamma}\right\rangle
			-\left\langle\partial^{\alpha} u^{\gamma} \cdot v \sqrt{\mu_i}, \partial^{\alpha} f_i^{\gamma}\right\rangle\\
			&-\left\langle\left(\frac{-|v|^{2}}{4}+\frac{3}{2}\dfrac{\bar{\theta}}{i}+\dfrac{\bar{\theta}^2}{i^2}\Delta_{v}\right) \partial^{\alpha} f_i^{\gamma}, \partial^{\alpha} f_i^{\gamma}\right\rangle 
			=\left|\partial^{\alpha}\left(u^{\gamma} \sqrt{\mu_i}-\frac{\bar{\theta}}{i} \nabla_{v} f_i^{\gamma}-\frac{v}{2} f_i^{\gamma}\right)\right|_{0}^{2}.
		\end{aligned}
		$$
		The notation $|\cdot|_{0}$ is defined in  \eqref{Sobolevnorm} with $s=r=0$, i.e., taking $L_{x, v}^{2}$ norm for a fixed $z$.
		
		The bad terms $B$ are given by
		$$
		\begin{aligned}
			B&=B_{1}+\dfrac{\kappa}{\epsilon}\sum_{i=1}^{N}i^{1/3}B_{2,i}+\dfrac{\kappa}{\epsilon}\sum_{i=1}^{N}i^{1/3}B_{3,i}\\
			&=\sum_{|\gamma| \leq r,|\alpha| \leq s} B_{1, \alpha, \gamma}+\dfrac{\kappa}{\epsilon}\sum_{i=1}^{N}i^{1/3}\sum_{|\gamma| \leq r,|\alpha| \leq s} B_{2, i,\alpha, \gamma}+\dfrac{\kappa}{\epsilon}\sum_{i=1}^{N}i^{1/3}\sum_{|\gamma| \leq r} B_{3, i,\gamma},
		\end{aligned}
		$$
		with
		$$
		\begin{aligned}
			&B_{1, \alpha, \gamma}=\left\langle\partial^{\alpha}\left(u \cdot \nabla_{x} u\right)^{\gamma}, \partial^{\alpha} u^{\gamma}\right\rangle, \\
			&B_{2,i,\alpha, \gamma}=\left\langle\partial^{\alpha}(u f_i)^{\gamma}, \partial^{\alpha}\left[u^{\gamma} \sqrt{\mu_i}-\frac{\bar{\theta}}{i} \nabla_{v} f_i^{\gamma}- \frac{v}{2} f_i^{\gamma}\right]\right\rangle, \\
			&B_{3,i, \gamma}=\frac{1}{|\mathbb{T}|^{3}}\left\langle(u f_i)^{\gamma}, \bar{u}^{\gamma} \sqrt{\mu_i}\right\rangle,
		\end{aligned}
		$$
		coming from the nonlinear terms.
		
		By using Lemma \ref{Lem1.3}, the bad terms are controlled by
		$$
		\begin{aligned}
			&\left|B_{1, \alpha, \gamma}\right| \leq \frac{C(s,r)}{\delta}|u|_{s+1, r}^{2}|u|_{s, r}^{2}+\delta|u|_{s, r}^{2} \leq \frac{C}{\delta} E G_{1}+\delta G_{1}, \\
			&\dfrac{\kappa}{\epsilon}\sum_{i=1}^{N}i^{1/3}\left|B_{2, i,\alpha, \gamma}\right| \leq \dfrac{\kappa}{\epsilon}\sum_{i=1}^{N}i^{1/3}\left(\frac{C(s,r)}{\delta}|u|_{s, r}^{2}|f_i|_{s, r}^{2}+\delta\left|u^{\gamma} \sqrt{\mu_i}-\frac{\bar{\theta}}{i} \nabla_{v} f_i^{\gamma}- \frac{v}{2} f_i^{\gamma}\right|_{s} \right)\leq \frac{C}{\delta} E G_{1}+\delta G_{2}, \\
			&\dfrac{\kappa}{\epsilon}\sum_{i=1}^{N}i^{1/3}\left|B_{3,i, \gamma}\right| \leq \dfrac{\kappa}{\epsilon}\sum_{i=1}^{N}i^{1/3}\left(\frac{C(s,r)}{\delta}|u|_{s, r}^{2}|f_i|_{s, r}^{2}+\delta\left|\bar{u}^{\gamma}\right|^{2}\right) \leq \frac{C}{\delta} E G_{1}+\delta G_{1}.
		\end{aligned}
		$$
		
		In conclusion, we have the energy estimate
		\begin{equation}\label{energyestimate}
			\frac{1}{2} \partial_{t} E \leq-(1-\frac{C}{\delta} E-C \delta) G.
		\end{equation}
		Take $\delta=\frac{1}{4 C}$ where $C$ is the constant in \eqref{energyestimate}, and $c_{1}(s, r)=\frac{1}{16C^2}$. Then we show that $E(t) \leq c_{1}$ for all t. In fact, let
		$$
		T^{*}=\sup \left\{\tilde{T} \geq 0: \sup _{0 \leq t<\tilde{T}} E(t) \leq c_{1}\right\}.
		$$
		Then it follows that $E(t) \leq c_{1}$ for $0 \leq t \leq T^{*}$. By the choice of $\delta$ and $c_{1}$, one has
		$$
		1-C(\delta) E-C \delta \geq 1-\frac{1}{4}-\frac{1}{4}=\frac{1}{2},
		$$
		and thus \eqref{energyestimate} implies
		\begin{equation}\label{ineq:E0}
			\partial_{t} E+G \leq 0,
		\end{equation}
		for $0 \leq t \leq T^{*}$. This prevents $T^{*}$ from being finite:
		Since $G>0$ and $E>0$, the estimate \eqref{ineq:E0} implies that $E(t)$ decreases and will not blow up. Therefore $T^*$ can be taken to infinity, a contradiction. Hence we proved $E(t) \leq c_{1}$ for all $t$ and thus \eqref{ineq:E0} holds for all $t$. Thus $E(t)$ is decreasing in $t$.
		\qed
	\end{proof}
	
	\subsection{Hypocoercivity estimates (Proof of Theorem \ref{thm:hypocoercivity})}
	
	\subsubsection{Preliminary work}
	
	For the proof of Theorem \ref{thm:hypocoercivity}, it is essential to give and prove the following lemma:
	\begin{lemma}\label{Lem2.1}
		There exisits a constant $C>0$ such that for $f_i(x, v) \in L_{x, v}^{2}$ orthogonal to $\sqrt{\mu_i}$, one has
		$$
		\|f_i\|_{L^{2}}^{2} \leq C\left(\|\mathcal{K}_i f_i\|_{L^{2}}^{2}+\epsilon^2\|\mathcal{S}_i f_i\|_{L^{2}}^{2}\right), \quad i=1,2,\ldots,N. 
		$$
	\end{lemma}
	
	\begin{proof}
		We argue by contradiction: Suppose that for any integer $n$, there exists a function $f_{i,n}$ such that
		\begin{equation}\label{eq:f1}
			\left\|f_{i,n}\right\|_{L^{2}}=1,
		\end{equation}
		and
		\begin{equation}\label{ineq:1n}
			\left\|\mathcal{K}_i f_{i,n}\right\|_{L^{2}}^{2}+\epsilon^2\left\|\mathcal{S}_i f_{i,n}\right\|_{L^{2}}^{2} \leq \frac{1}{n} .
		\end{equation}
		Recalling that
		\begin{equation}\label{eq:Kf}
			\left\|\dfrac{\bar{\theta}}{i}\nabla_{v} f_i\right\|_{L^{2}}^{2}+\left\|\frac{v}{2} f_i\right\|_{L^{2}}^{2}=\|\mathcal{K}_i f_i\|_{L^{2}}^{2}+\frac{3\bar{\theta}}{2i}\|f_i\|_{L^{2}}^{2},
		\end{equation}
		a combination of \eqref{eq:f1}- \eqref{eq:Kf} yields
		$$
		\left\|\dfrac{\bar{\theta}}{i}\nabla_{v} f_{i,n}\right\|_{L^{2}}^{2}+\left\|\frac{v}{2} f_{i,n}\right\|_{L^{2}}^{2}+\epsilon^2\dfrac{\bar{\theta}^2}{i^2}\left\|\nabla_{x} f_{i,n}\right\|_{L^{2}}^{2} \leq \frac{1}{n}+\frac{3\bar{\theta}}{2i}.
		$$
		Since this estimate controls both the derivatives of $f_{i,n}$ and the tails for large velocities, as a consequence of the Rellich-Kondrazhov theorem,
		we can assume that a subsequence satisfies
		$$
		\begin{array}{rr}
			f_{i,n_{k}} \longrightarrow f_i & \text { strongly in } L^{2}\left(\mathbb{T}^{3} \times \mathbb{R}^{3}\right), \\
			\nabla_{x} f_{i,n_{k}} \rightarrow \nabla_{x} f_i \text { and } \nabla_{v} f_{i,n_{k}} \rightarrow \nabla_{v} f_i & \text { weakly in } L^{2}\left(\mathbb{T}^{3} \times \mathbb{R}^{3}\right),
		\end{array}
		$$
		with furthermore $\|f_i\|_{L^{2}}=1$. Coming back to \eqref{ineq:1n}, one obtains
		$$
		\left\|\dfrac{\bar{\theta}}{i}\nabla_{v} f_i+\frac{v}{2} f_i\right\|_{L^{2}}^{2}+\epsilon^2\dfrac{\bar{\theta^2}}{i^2}\left\|\nabla_{x} f_i\right\|_{L^{2}}^{2} \leq \liminf _{k \rightarrow \infty}\left(\left\|\mathcal{K}_i f_{i,n_{k}}\right\|_{L^{2}}^{2}+\epsilon^2\left\|\mathcal{S}_i f_{i,n_{k}}\right\|_{L^{2}}^{2}\right)=0.
		$$
		One deduces that $f_i(x, v)=\tau \sqrt{\mu_i(v)}$ for some $\tau \in \mathbb{R}$. Finally, assuming that the $f_{i,n}$'s are orthogonal to $\sqrt{\mu_i}$, one gets
		$$
		\int_{\mathbb{T}^{3} \times \mathbb{R}^{3}} f_i \sqrt{\mu_i} \mbox{d} v \mbox{d} x=\lim _{k \rightarrow \infty} \int_{\mathbb{T}^{3} \times \mathbb{R}^{3}} f_{i,n_{k}} \sqrt{\mu_i} \mbox{d} v \mbox{d} x=0.
		$$
		Hence $f_i = 0$, which contradicts to the fact that $f_i$ is normalized. This completes the
		proof.
		\qed
	\end{proof}
	
	Then we prove the following lemma, which is an analog to Proposition $4.1$ of \cite{GoudonHe2010}:
	
	\begin{lemma}\label{Lem2.2}
		Let the assumptions of Theorem \ref{thm:hypocoercivity} be fulfilled. Then there exists a constant $c_1'(s, r) \leq c_{1}(s+3, r)$ such that, if we assume that $E_{s+3, r}(0) \leq c_{1}^{\prime}(s, r)$ is small enough, then there exists a constant $\lambda_{1}>0$ such that $($at each $z)$
		\begin{equation}\label{ineq:flambda1}
			\partial_{t}(f_i, f_i)_{s, r}+\frac{\lambda_{1}}{\bar{\theta}\epsilon^2} [f_i, f_i]_{s, r} \leq \frac{C(\lambda_{1})}{\bar{\theta}}\left(|u|_{s, r}^{2}+\left|\nabla_{x} u\right|_{s, r}^{2}+\dfrac{1}{\epsilon^2}|\mathcal{K}_i f_i|_{s, r}^{2}\right).
		\end{equation}
	\end{lemma}
	
	\begin{proof}
		One can write the evolution equation of $\partial^{\alpha} f_i^{\gamma}$ as
		\begin{equation}\label{eq:fgamma}
			\begin{aligned}
				\partial_{t} \partial^{\alpha} f_i^{\gamma}+\dfrac{i}{\bar{\theta}}\mathcal{P}_i \partial^{\alpha} f_i^{\gamma}+& \frac{i^{1/3}}{\bar{\theta}\epsilon}\left(\mathcal{K}_i^{*} \cdot \mathcal{K}_i\right) \partial^{\alpha} f_i^{\gamma}= \frac{i^{1/3}}{\bar{\theta}\epsilon} \partial^{\alpha} u^{\gamma} \cdot v \sqrt{\mu_i} \\
				&+\frac{i^{1/3}}{\bar{\theta}\epsilon} \sum_{0 \leq \eta \leq \alpha} \sum_{0 \leq \beta \leq \gamma}\left(\begin{array}{l}
					\gamma \\
					\beta
				\end{array}\right)\left(\begin{array}{l}
					\alpha \\
					\eta
				\end{array}\right) \partial^{\eta} u^{\beta} \cdot \mathcal{K}_i^{*} \partial^{\alpha-\eta} f_i^{\gamma-\beta}.
			\end{aligned}
		\end{equation}
		Take the $(\cdot, \cdot)$ inner product of \eqref{eq:fgamma} with $\partial^{\alpha} f_i^{\gamma}$. 
		
		For the linear terms, by the same argument as the proof of Proposition $4.1$ of \cite{GoudonHe2010}, one gets
		$$
		\begin{aligned}
			i^{1/3}\left(\mathcal{P}_i \partial^{\alpha} f_i^{\gamma}, \partial^{\alpha} f_i^{\gamma}\right)=& 2 \left\langle\mathcal{S}_i \partial^{\alpha} f_i^{\gamma}, \mathcal{K}_i \partial^{\alpha} f_i^{\gamma}\right\rangle+\epsilon^2\left|\mathcal{S}_i \partial^{\alpha} f_i^{\gamma}\right|_{0}^{2} \geq \frac{3}{4}\epsilon^2\left|\mathcal{S}_i \partial^{\alpha} f_i^{\gamma}\right|_{0}^{2}- \frac{4}{\epsilon^2} \left|\mathcal{K}_i \partial^{\alpha} f_i^{\gamma}\right|_{0}^{2}, \\
			\dfrac{i^{1/3}}{\epsilon}\left(\mathcal{K}_i^{*} \cdot \mathcal{K}_i \partial^{\alpha} f_i^{\gamma}, \partial^{\alpha} f_i^{\gamma}\right)=& \dfrac{2}{\epsilon} \left|\mathcal{K}_i \partial^{\alpha} f_i^{\gamma}\right|_{0}^{2}+\dfrac{2}{\epsilon} \left|\mathcal{K}_i^{2} \partial^{\alpha} f_i^{\gamma}\right|_{0}^{2} +\epsilon\left\langle\mathcal{K}_i \partial^{\alpha} f_i^{\gamma}, \mathcal{S}_i \partial^{\alpha} f_i^{\gamma}\right\rangle\\
			&+2\epsilon \left\langle\mathcal{K}_i^{2} \partial^{\alpha} f_i^{\gamma}, \mathcal{S}_i \mathcal{K}_i \partial^{\alpha} f_i^{\gamma}\right\rangle +\epsilon^2\left|\mathcal{S}_i \mathcal{K}_i \partial^{\alpha} f^{\gamma}\right|_{0}^{2}\\
			\geq & \dfrac{3}{2} \left|\mathcal{K}_i \partial^{\alpha} f_i^{\gamma}\right|_{0}^{2}-\frac{\epsilon^2}{2}\left|\mathcal{S}_i \partial^{\alpha} f^{\gamma}\right|_{0}^{2}+\frac{1}{2} \left|\mathcal{K}_i^{2} \partial^{\alpha} f^{\gamma}\right|_{0}^{2}+\frac{\epsilon^2}{3}\left|\mathcal{S}_i \mathcal{K}_i \partial^{\alpha} f^{\gamma}\right|_{0}^{2}, \\
			\dfrac{i^{1/3}}{\epsilon}\left|\left(\partial^{\alpha} u^{\gamma} \cdot v \sqrt{\mu_i}, \partial^{\alpha} f_i^{\gamma}\right)\right|=& \left| \dfrac{2}{\epsilon} \left\langle\mathcal{K}_i\left(\partial^{\alpha} u^{\gamma} \cdot v \sqrt{\mu_i}\right), \mathcal{K}_i \partial^{\alpha} f_i^{\gamma}\right\rangle+\epsilon\left\langle\mathcal{K}_i\left(\partial^{\alpha} u^{\gamma} \cdot v \sqrt{\mu_i}\right), \mathcal{S}_i \partial^{\alpha} f_i^{\gamma}\right\rangle \right. \\
			&\left. +\epsilon\left\langle\mathcal{S}_i\left(\partial^{\alpha} u^{\gamma} \cdot v \sqrt{\mu_i}\right), \mathcal{K}_i \partial^{\alpha} f_i^{\gamma}\right\rangle+\epsilon^2\left\langle\mathcal{S}_i\left(\partial^{\alpha} u^{\gamma} \cdot v \sqrt{\mu_i}\right), \mathcal{S}_i \partial^{\alpha} f_i^{\gamma}\right\rangle \right| \\
			\leq & \delta\left(\dfrac{1}{\epsilon^2}\left|\mathcal{K}_i \partial^{\alpha} f_i^{\gamma}\right|_{0}^{2}+\epsilon^2\left|\mathcal{S}_i \partial^{\alpha} f_i^{\gamma}\right|_{0}^{2}\right)+\frac{C_1}{\delta}\left(|u|_{s, r}^{2}+\left|\nabla_{x} u\right|_{s, r}^{2}\right) .
		\end{aligned}
		$$
		The notation $|\cdot|_{0}$ is defined by \eqref{Sobolevnorm} with $s=r=0$, i.e., taking $L_{x, v}^{2}$ norm for a fixed $z$. 
		
		For the nonlinear term (the summation), one gets
		$$
		\begin{aligned}
			& \dfrac{i^{1/3}}{\epsilon}\left|\left(\partial^{\eta} u^{\beta} \cdot \mathcal{K}_i^{*} \partial^{\alpha-\eta} f_i^{\gamma-\beta}, \partial^{\alpha} f_i^{\gamma}\right)\right| \\ =&
			\left| \frac{2}{\epsilon}\left\langle \partial^{\eta} u^{\beta} \cdot K^{*} \partial^{\alpha-\eta} f_i^{\gamma-\beta}, K^{2} \partial^{\alpha} f_i^{\gamma}\right\rangle+\frac{2}{\epsilon}\left\langle \partial^{\eta} u^{\beta} \partial^{\alpha-\eta} f_i^{\gamma-\beta}, K \partial^{\alpha} f_i^{\gamma}\right\rangle\right.\\
			&+\epsilon\left\langle K\left(\partial^{\eta} u^{\beta} \cdot K^{*} \partial^{\alpha-\eta} f_i^{\gamma-\beta}\right), S \partial^{\alpha} f_i^{\gamma}\right\rangle 
			+\epsilon\left\langle S\left(\partial^{\eta} u^{\beta} \cdot K^{*} \partial^{\alpha-\eta} f_i^{\gamma-\beta}\right), K \partial^{\alpha} f_i^{\gamma}\right\rangle\\
			&\left.+\epsilon^2\left\langle S\left(\partial^{\eta} u^{\beta} \cdot K^{*} \partial^{\alpha-\eta} f_i^{\gamma-\beta}\right), S \partial^{\alpha} f_i^{\gamma}\right\rangle \right| \\
			\leq & C_2 \left|\partial^{\eta} u^{\beta}\right|_{3,0}\frac{1}{\epsilon^2}\left(\left[\partial^{\alpha-\eta} f_i^{\gamma-\beta}, \partial^{\alpha-\eta} f_i^{\gamma-\beta}\right]+\left[\partial^{\alpha} f_i^{\gamma}, \partial^{\alpha} f_i^{\gamma}\right]\right) \\
			\leq & C_2  \frac{1}{\epsilon^2} |u|_{s+3, r}[f_i, f_i]_{s, r},
		\end{aligned}
		$$
		where we used the fact that the $x$ and $z$ derivatives commute with the operators $\mathcal{K}_i$ and $\mathcal{S}_i$. In fact,
		for the term $\dfrac{i^{1/3}}{\epsilon}\left(u\cdot \mathcal{K}_i^*f_i,g_i\right)$, one has
		\begin{equation}\label{eq:Kfg}
			\begin{aligned}
				&\dfrac{i^{1/3}}{\epsilon}\left(u\cdot K^*f_i,g_i\right)\\=& \frac{2}{\epsilon}\left\langle\mathcal{K}_i\left(u \cdot \mathcal{K}_i^{*} f_i\right), \mathcal{K}_i g_i\right\rangle+\epsilon\left\langle\mathcal{K}_i\left(u \cdot \mathcal{K}_i^{*} f_i\right), \mathcal{S}_i g_i\right\rangle+\epsilon\left\langle\mathcal{S}_i\left(u \cdot \mathcal{K}_i^{*} f_i\right), \mathcal{K}_i g_i\right\rangle+\epsilon^{2}\left\langle\mathcal{S}_i\left(u \cdot \mathcal{K}_i^{*} f_i\right), \mathcal{S}_i g_i\right\rangle \\
				=& \frac{2}{\epsilon}\left\langle u \mathcal{K}_i f_i, \mathcal{K}_i^{2} g_i\right\rangle+\frac{2}{\epsilon}\langle u f_i, \mathcal{K}_i g_i\rangle+\epsilon\langle u \mathcal{K}_i f_i, \mathcal{K}_i \mathcal{S}_i g_i\rangle+\epsilon\langle u f_i, \mathcal{S}_i g_i\rangle +\epsilon\left\langle\mathcal{S}_i(u f_i), \mathcal{K}_i^{2} g_i\right\rangle+\epsilon^{2}\langle\mathcal{S}_i(u f_i), \mathcal{S}_i \mathcal{K}_i g_i\rangle \\
				=& \frac{2}{\epsilon}\left\langle u \mathcal{K}_i f_i, \mathcal{K}_i^{2} g_i\right\rangle+\frac{2}{\epsilon}\langle u f_i, \mathcal{K}_i g_i\rangle+\epsilon\langle u \mathcal{K}_i f_i, \mathcal{K}_i \mathcal{S}_i g_i\rangle+\epsilon\langle u f_i, \mathcal{S}_i g_i\rangle +\epsilon\left\langle(\mathcal{S}_i u) f_i, \mathcal{K}_i^{2} g_i\right\rangle+\epsilon\left\langle u(\mathcal{S}_i f_i), \mathcal{K}_i^{2} g_i\right\rangle \\
				&+\epsilon^{2}\langle(\mathcal{S}_i u) f_i, \mathcal{S}_i \mathcal{K}_i g_i\rangle+\epsilon^{2}\langle u(\mathcal{S}_i f_i), \mathcal{K}_i \mathcal{S}_i g_i\rangle .
			\end{aligned}
		\end{equation}
		For each term in \eqref{eq:Kfg}, by using the Cauchy-Schwarz inequality, Lemma \ref{Lem2.1}, and the Sobolev inequality
		$\|u\|_{L^{\infty}}+\left\|\nabla_{x} u\right\|_{L^{\infty}} \leq C\|u\|_{H^{3}},
		$
		one has
		$$
		\begin{aligned}
			\frac{1}{\epsilon}\left\langle u \mathcal{K}_i f_i, \mathcal{K}_i^{2} g_i\right\rangle & \leq  \dfrac{1}{\epsilon}\|u\|_{L^{\infty}}\|\mathcal{K}_i f_i\|_{L^{2}}\|\mathcal{K}_i^2 g_i\|_{L^{2}} \leq C\|u\|_{L^{\infty}}\left(\dfrac{1}{\epsilon^2}\|\mathcal{K}_i f_i\|_{L^{2}}^2+\|\mathcal{K}_i^2 g_i\|_{L^{2}}^2\right), \\
			\dfrac{1}{\epsilon}\langle u f_i, \mathcal{K}_i g_i\rangle & \leq \dfrac{1}{\epsilon}\|u\|_{L^{\infty}}\|f_i\|_{L^{2}}\|\mathcal{K}_i g_i\|_{L^{2}} \leq C\dfrac{1}{\epsilon}\|u\|_{L^{\infty}}\left(\|\mathcal{K}_i f_i\|_{L^{2}}+\epsilon^2\|\mathcal{S}_i f\|_{L^{2}}\right)\|\mathcal{K}_i g_i\|_{L^{2}} \\
			& \leq C\|u\|_{L^{\infty}}\left(\dfrac{1}{\epsilon^2}\|\mathcal{K}_i f_i\|_{L^{2}}^{2}+\|\mathcal{K}_i g\|_{L^{2}}^{2}+\epsilon^2\|\mathcal{S}_i f\|_{L^{2}}^{2}+\|\mathcal{K}_i g_i\|_{L^{2}}^{2}\right), \\
			\epsilon\langle u \mathcal{K}_i f_i, \mathcal{K}_i \mathcal{S}_i g_i\rangle & \leq \epsilon\|u\|_{L^{\infty}}\|\mathcal{K}_i f_i\|_{L^{2}}\|\mathcal{K}_i\mathcal{S}_i g_i\|_{L^{2}} \leq C\|u\|_{L^{\infty}}\left(\|\mathcal{K}_i f_i\|_{L^{2}}^2+\epsilon^2\|\mathcal{K}_i\mathcal{S}_i  g_i\|_{L^{2}}^2\right),\\
			\epsilon\langle u f_i, \mathcal{S}_i g_i\rangle & \leq \epsilon\|u\|_{L^{\infty}}\|f_i\|_{L^{2}}\|\mathcal{S}_i g_i\|_{L^{2}} \leq C \epsilon\|u\|_{L^{\infty}}\left(\|\mathcal{K}_i f_i\|_{L^{2}}+\epsilon^2\|\mathcal{S}_i f_i\|_{L^{2}}\right)\|\mathcal{S}_i g_i\|_{L^{2}} \\
			& \leq C \epsilon\|u\|_{L^{\infty}}\left(\dfrac{1}{\epsilon^2}\|\mathcal{K}_i f_i\|_{L^{2}}^{2}+\epsilon^2\|\mathcal{S}_i g_i\|_{L^{2}}^{2}+\epsilon^2\|\mathcal{S}_i f_i\|_{L^{2}}^{2}+\epsilon^2\|\mathcal{S}_i g_i\|_{L^{2}}^{2}\right).
		\end{aligned}
		$$
		The latter four terms in \eqref{eq:Kfg} can be approximated similarly and one obtains
		$$
		\frac{i^{1/3}}{\epsilon}\left|\left(u \cdot \mathcal{K}_i^{*} f_i, g_i\right)\right| \leq C \frac{1}{\epsilon^2}\|u\|_{H^{3}}([f_i, f_i]+[g_i, g_i]).
		$$
		
		A combination of the above estimates yields
		$$
		\begin{aligned}
			&\frac{1}{2} \partial_{t}\left(\partial^{\alpha} f_i^{\gamma}, \partial^{\alpha} f_i^{\gamma}\right)+\dfrac{1}{\bar{\theta}}\left( \dfrac{\epsilon^2}{4}\left|\mathcal{S}_i \partial^{\alpha} f_i^{\gamma}\right|_{0}^{2} + \dfrac{1}{2}\left|\mathcal{K}_i^{2} \partial^{\alpha} f_i^{\gamma}\right|_{0}^{2}+\frac{\epsilon^2}{3}\left|\mathcal{S K} \partial^{\alpha} f_i^{\gamma}\right|_{0}^{2}-\left(\frac{4}{\epsilon^2}-\frac{3}{2}\right)\left|\mathcal{K}_i \partial^{\alpha} f_i^{\gamma}\right|_{0}^{2}\right)\\
			&\quad \leq \frac{\delta}{\bar{\theta}}\left(\frac{1}{\epsilon^2}\left|\mathcal{K}_i \partial^{\alpha} f_i^{\gamma}\right|_{0}^{2}+\epsilon^2\left|\mathcal{S}_i \partial^{\alpha} f_i^{\gamma}\right|_{0}^{2}\right)+\frac{C_1}{\delta\bar{\theta}}\left(|u|_{s, r}^{2}+\left|\nabla_{x} u\right|_{s, r}^{2}\right)+ \frac{C_2}{\bar{\theta}\epsilon^2}|u|_{s+3, r}[f_i, f_i]_{s, r}.
		\end{aligned}
		$$
		Summing over $\alpha, \gamma$, and choosing $\delta=\frac{1}{8} $ to absorb the term $\left|\mathcal{S}_i \partial^{\alpha} f_i^{\gamma}\right|_{0}^{2}$ on the RHS by the same term on the LHS,  one gets
		$$
		\partial_{t}(f_i, f_i)_{s, r}+\frac{1}{\bar{\theta}\epsilon^2}\left(\frac{1}{8}-MC_2|u|_{s+3, r}\right) [f_i, f_i]_{s, r} \leq \frac{C_3}{\bar{\theta}}\left(|u|_{s, r}^{2}+\left|\nabla_{x} u\right|_{s, r}^{2}+\frac{1}{\epsilon^2}|\mathcal{K}_i f_i|_{s, r}^{2}\right),
		$$
		where 
		$M$ is the number of possible pairs $(\alpha, \gamma)$, $C_3 = \max \left\{5, 8M C_1\right\}$.
		Thus if one chooses $c_{1}^{\prime}=\min \left\{c_{1}(s+3, r), \frac{1}{16 MC_2}\right\}$, then by Theorem \ref{thm:energyestimate}, $E_{s+3, r}(t)$ is decreasing and $E_{s+3, r}(t) \leq$ $c_{1}^{\prime}$ for all $t$. Hence $|u|_{s+3, r} \leq E_{s+3, r} \leq c_{1}^{\prime}$ for all $t$ and one gets the conclusion with $\lambda_{1}=\frac{1}{16}$.\qed
	\end{proof}
	
	\subsubsection{Proof of Theorem \ref{thm:hypocoercivity}}
	
	\begin{proof}
		Multiplying \eqref{ineq:flambda1} by $\lambda_{4}\kappa\bar{\theta}$ ($\lambda_{4}>0$ is a constant to be chosen later), summing over $i$, and adding to \eqref{ineq:E0} yields
		$$
		\partial_{t} \tilde{E}+\tilde{G} \leq \lambda_{4} \tilde{B},
		$$
		where
		$$
		\begin{aligned}
			&\tilde{E}=E+\lambda_{4}\kappa\bar{\theta}\sum_{i=1}^{N}(f_i, f_i)_{s, r}, \quad \tilde{G}=G+\lambda_{4} \lambda_{1}\frac{\kappa}{\epsilon^2} \sum_{i=1}^{N}[f_i, f_i]_{s, r},  \\
			&\tilde{B}=C\left(\lambda_{1}\right)\kappa\sum_{i=1}^{N}\left(|u|_{s, r}^{2}+\left|\nabla_{x} u\right|_{s, r}^{2}+\frac{1}{\epsilon^{2}}|\mathcal{K}_i f_i|_{s, r}^{2}\right) . 
		\end{aligned}
		$$
		It is clear that 
		$$\tilde{B} \leq C\left(G+\kappa\sum_{i=1}^{N}\frac{1}{\epsilon^{2}}|\mathcal{K}_i f_i|_{s, r}^{2}\right)
		\leq C \tilde{G}.$$ Thus by choosing $\lambda_{4}=\min \left\{\frac{1}{2 C}, 1\right\}, C$ being the previous constant, one gets
		\begin{equation}\label{ineq:Etilde0}
			\partial_{t} \tilde{E}+\frac{1}{2} \tilde{G} \leq 0.
		\end{equation}
		Notice that Lemma \ref{Lem2.1} implies that
		$$
		|f_i|_{s, r}^{2} \leq C\left(|\mathcal{K}_i f_i|_{s, r}^{2}+\epsilon^2|\mathcal{S}_i f_i|_{s, r}^{2}\right),
		$$
		and by definition one also has
		$$
		(f_i, f_i)_{s, r} \leq C\left(|\mathcal{K}_i f_i|_{s, r}^{2}+\epsilon^2|\mathcal{S}_i f_i|_{s, r}^{2}\right) \leq C \frac{1}{\epsilon^{2}}[f_i, f_i]_{s, r}.
		$$
		Thus
		\begin{equation}\label{ineq:EtildeG}
			\tilde{E} \leq C\left(G+\kappa\bar{\theta}\sum_{i=1}^{N}|f_i|_{s, r}^{2}\right)+\lambda_{4}\kappa\bar{\theta}\sum_{i=1}^{N}(f_i, f_i)_{s, r} \leq C\left(G+\sum_{i=1}^{N}\left(|\mathcal{K}_i f_i|_{s, r}^{2}+\epsilon^2|\mathcal{S}_i f_i|_{s, r}^{2}\right)\right) \leq C \tilde{G}.
		\end{equation}
		This together with \eqref{ineq:Etilde0} implies
		$$
		\tilde{E}(t) \leq \tilde{E}(0) e^{-\lambda t},
		$$
		where $\lambda=\frac{1}{2 C}$, $C$ being the constant in \eqref{ineq:EtildeG}.
		Finally, the proof of Theorem \ref{thm:hypocoercivity} is completed by noticing that
		$$
		E(t) \leq \tilde{E}(t) \leq \tilde{E}(0) e^{-\lambda t} \leq\left(E(0)+C^{h}\right) e^{-\lambda t}.
		$$ \qed
	\end{proof}

	\section{Spectral accuracy of the gPC-sG approximation}
	
	\subsection{Notations and preliminary results}
	
	We then introduce the gPC-sG method for the multi-phase flow model \eqref{eq:uf}. Take the basis functions $\left\{\phi_{k}(z)\right\}_{k=1}^{\infty}$ as the gPC basis, i.e., the set of polynomials defined on $\mathbb{Z}$, orthonormal with respect to the given probability measure $\pi(z) \mathrm{d} z$, with $\phi_{k}$ being a polynomial of degree $k-1$.
	Expand the functions $u, \{f_i\}_{i=1}^N$ and $p$ into
	$$
	u(t, x, z)=\sum_{k=1}^{\infty} u_{k}(t, x) \phi_{k}(z), \quad f_i(t, x, v, z)=\sum_{k=1}^{\infty} f_{ik}(t, x, v) \phi_{k}(z), \quad p(t, x, z)=\sum_{k=1}^{\infty} p_{k}(t, x) \phi_{k}(z),
	$$
	and approximate them by truncated series up to order $K$ :
	$$
	u \approx u^{K}=\sum_{k=1}^{K} u_{k} \phi_{k}(z), \quad f_i \approx f_i^{K}=\sum_{k=1}^{K} f_{ik} \phi_{k}(z), \quad p \approx p^{K}=\sum_{k=1}^{K} p_{k} \phi_{k}(z).
	$$
	Then by substituting into \eqref{eq:uf} and conducting the Galerkin projection, one gets the following deterministic system for $(u_{k}, \{f_{ik}\}_{i=1}^N)_{k=1}^{K}$ :
	\begin{equation}\label{eq:ufk}
		\left\{\begin{aligned}
			&(f_{ik})_{t}+v \cdot \nabla_{x} (f_{ik})+\frac{1}{i^{2/3}\epsilon} \cdot\left(\nabla_{v} -\frac{ i v}{2 \bar{\theta}} \right)(uf_i)_k-\frac{i^{1/3}}{\bar{\theta}\epsilon}u_k \cdot v \sqrt{\mu_i}=\frac{1}{i^{2/3}\epsilon}\left(-\frac{i}{\bar{\theta}}\frac{|v|^{2}}{4} +\frac{3}{2} +\frac{\bar{\theta}}{i}\Delta_{v} \right)f_{ik}, \\
			&(u_k)_{t}+(u \cdot \nabla_{x} u)_k+\nabla_{x} p_k-\Delta_{x} u_k+\frac{\kappa}{\epsilon}\sum_{i=1}^{N} i^{1/3} u_k +\frac{\kappa}{\epsilon}  \sum_{i=1}^{N}i^{1/3}  \int_{\mathbb{R}^{3}} \sqrt{\mu_i} (uf_i)_{k}  \mbox{d} v\\
			&\hspace{27em}-\frac{\kappa}{\epsilon}\sum_{i=1}^{N} i^{1/3}  \int_{\mathbb{R}^{3}} v \sqrt{\mu_i} f_{ik} \mbox{d} v=0, \\
			&\nabla_{x} \cdot u_k=0
		\end{aligned}	
		\right.
	\end{equation}
	with the initial data
	$$
	\left.u_{k}\right|_{t=0}=(u_{0})_{k}=\int u_{0} \phi_{k}(z) \pi(z) \mathrm{d} z,\left.\quad f_{ik}\right|_{t=0}=(f_{i,0})_{k}, i=1,2,\ldots,N.
	$$
	Here the gPC coefficient of a product is given by
	$$
	(u w)_{k}=\sum_{j, l=1}^{K} S_{j l k} u_{j} w_{l}
	$$
	where
	\begin{equation}\label{Sijk}
		S_{j l k}=\int \phi_{j} \phi_{l} \phi_{k} \pi(z) \mathrm{d} z,
	\end{equation}
	is the triple product coefficient.
	
	\subsection{Main Results}
	
	The goal is to show that under $\epsilon$-independent smallness assumptions on initial data, the gPC-sG method \eqref{eq:ufk} has uniform-in-$\epsilon$ spectral accuracy for all $K$. 
	Although the system \eqref{eq:ufk} is similar to the original system \eqref{eq:uf}, it indeed requires some $K$-indepnedent smallness requirement on initial data to obtain an estimate independent of $K$. The difficulty comes from the $K^{2}$ nonlinear terms appeared in the $\mathrm{gPC}$ product \eqref{Sijk}.
	
	To overcome this difficulty, we introduce the technical condition \eqref{ineq:thmgPC1} and the weighted Sobolev norm $\sum_{k=1}^{K}\left\|k^{q} u_{k}\right\|_{s}^{2}$ (see the theorem below for detail) as in \cite{ShuJin2018}. 
	It is natural to approximate a nonlinear estimate with this weighted Sobolev norm, being independent of $K$, similar to $\|u w\|_{H^{s}} \leq C\|u\|_{H^{s}}\|w\|_{H^{s}}$ in the $x$-space. With the aid of this technique, we prove
	
	\begin{theorem}\label{thm:gPCcoefficients}
		Assume the technical condition
		\begin{equation}\label{ineq:thmgPC1}
			\left\|\phi_{k}\right\|_{L^{\infty}} \leq C k^{p}, \quad \forall k,
		\end{equation}
		with a parameter $p>0$. Let $q>p+2$ and $s \geq 2$. Let $\left(u_{k}, f_{ik}\right), k=1, \ldots, K$, solve \eqref{eq:ufk} with initial data verifying \eqref{initialu0}, and define the energy $E^{K}$ by
		$$
		E^{K}(t)=E_{s, q}^{K}(t)=\sum_{k=1}^{K}\left(\left\|k^{q} u_{k}\right\|_{s}^{2}+\kappa\bar{\theta}\sum_{i=1}^N\left\|k^{q} f_{ik}\right\|_{s}^{2}+\left|k^{q} \bar{u}_{k}\right|^{2}\right).
		$$
		Then there exists a constant $c_{2}=c_{2}(s, q)>0$, independent of $K$, such that $E^{K}(0) \leq c_{2}$ implies that $E^{K}(t)$ is decreasing in $t$.
	\end{theorem}
	
	By introducing the technical condition \eqref{ineq:thmgPC1} and the weighted Sobolev norm $\sum_{k=1}^{K}\left\|k^{q} u_{k}\right\|_{s}^{2}$, Theorem \ref{thm:gPCcoefficients} is  proved by the same type of energy estimate as Theorem \ref{thm:energyestimate}. 
	It is noted that $c_{2}$ being independent of $K$ is important because 
	this implies from Proposition \ref{prop_gPC1}
	that the condition $E^{K}(0) \leq c_{2}$ is in fact a consequence of a smoothness condition on $\left(u_{0}, \{f_{i,0}\}_{i=1}^N\right)$ for all $K$, which means the gPC-sG method is stable for all $K$.
	
	Next we directly give the following two propositions as in \cite{ShuJin2018} and proofs of the two propositions are ommited.
	Proposition \ref{prop_gPC1} shows a sufficient condition on the initial data, under which the assumption $E^{K}(0) \leq c_{2}$ in Theorem \ref{thm:gPCcoefficients} holds. 
	Proposition \ref{prop_gPC2} shows that \eqref{ineq:thmgPC1} holds for gPC basis with respect to a large class of probability measures supported on a finite interval.
	
	\begin{proposition} \label{prop_gPC1}
		With the same assumptions as Theorem \ref{thm:gPCcoefficients}, the condition $E_{s, q}^{K}(0) \leq c_{2}(s, q)$ holds if $\left\|E_{s, r}(0)\right\|_{L_{z}^{1}} \leq C c_{2}(s, q)$ with $r>q+\frac{1}{2}$, and $C=C(s, q, r)$.
	\end{proposition}

	\begin{proposition}\label{prop_gPC2}
		Suppose $\mathbb{Z}=[-R, R], R<+\infty$ with $\pi(z)$ satisfying $1 / \pi(z) \in L^{p_{1}}$ for some $p_{1}>0$. Then \eqref{ineq:thmgPC1} holds with $p=1+1 / p_{1}$.
	\end{proposition}

	Note that if $\pi(z)$ is continuous and has only finite number of zeros, with $\pi\left(z-z_{0}\right) \geq c\left|z-z_{0}\right|^{p_{3}}$ for some $p_{3}>0$, $c>0$ near any zero $z=z_{0}$, then the condition of Proposition \ref{prop_gPC2} is satisfied with any $p_{1}<1 / p_{3}$. This includes all piecewise polynomial probability distributions on a finite interval with isolated zeros. 
	For some special cases \cite{Szego1939}, 
	\eqref{ineq:thmgPC1} holds with $p=1 / 2$ for the uniform distribution on $\mathbb{Z}=[-1,1]$ with normalized Legendre polynomials as gPC basis, and
	holds with $p=0$ for the distribution $\pi(z)=\frac{2}{\pi \sqrt{1-z^{2}}}$ on $\mathbb{Z}=[-1,1]$ with normalized Chebyshev polynomials as gPC basis. 
	
	Finally, 
	we obtain the spectral accuracy of the gPC-sG method, uniformly in $t$ and $\epsilon$, with a small initial data assumption on $\left(u_{0}, \{f_{i,0}\}_{i=1}^N\right)$, independent of $K$ and $\epsilon$:

	\begin{theorem}\label{thm:gPCaccuracy}
		Assume \eqref{ineq:thmgPC1} holds. Let $(u_{k}, \{f_{ik}\}_{i=1}^N)_{k=1}^K$ solve \eqref{eq:ufk} with initial data verifying \eqref{initialu0}-\eqref{initialu0_2}. There exists a constant $c_{1}^{\prime \prime}(s, r)$ such that the following holds: Assume $s \geq 0$, $r>p+\frac{5}{2}$, $\left\|E_{s+3, r}(0)\right\|_{L_{z}^{\infty}} \leq c_{1}^{\prime \prime}(s, r)$, and $C_{s, r}^{h}$ is finite. Then $E^{e}$, the energy of the gPC approximation error, defined $b y$
		$$
		E^{e}=\left\|u^{e}\right\|_{s}^{2}+\kappa\bar{\theta}\sum_{i=1}^N\left\|f_i^{e}\right\|_{s}^{2}+\left\|\bar{u}^{e}\right\|^{2}, \quad u^{e}=u-u^{K}, \quad f_i^{e}=f_i-f_i^{K},
		$$
		satisfies
		$$
		E^{e} \leq \frac{C}{K^{2 r}},
		$$
		for all time, i.e., the gPC-sG method has $r$-th order accuracy uniformly in time.
	\end{theorem}
	
	This theorem is proved by an energy estimate in the $(x, v, z)$ space on $\left(u^{e},\{ f_i^{e}\}_{i=1}^N\right)$ with the aid of the previous theorems.
	
	Finally we prove that the error also decays exponentially in time by a hypocoercivity argument:

	\begin{theorem}\label{thm:gPChypocoercivity}
		Assume \eqref{ineq:thmgPC1} holds. Let $(u_{k}, \{f_{ik}\}_{i=1}^N)_{k=1}^K$ solves \eqref{eq:ufk} with initial data verifying \eqref{initialu0}-\eqref{initialu0_2}. There exists a constant $c_{2}^{\prime \prime}(s, r)$ such that the following holds: Assume $s \geq 0, r>p+\frac{5}{2}$, $\left\|E_{s+6, r}(0)\right\|_{L_{z}^{\infty}} \leq c_{2}^{\prime \prime}(s, r)$, and $C_{s+3, r}^{h}$ is finite. Then there exists a constant $\lambda^{e}>0$ such that
		$$
		E^{e} \leq \frac{C}{K^{r-p-1 / 2}} e^{-\lambda^{e} t}.
		$$
	\end{theorem}
	
	These theorems imply that for random initial data near the global equilibrium, in the sense that $\left(u_{0}, \{f_{i,0}\}_{i=1}^N\right)$ is small in some suitable Sobolev spaces, the gPC-sG method has spectral accuracy, uniformly in time and $\epsilon$, and it captures the long-time behavior of \eqref{eq:uf} with random initial data.

	\subsection{Estimate of the gPC coefficients (Proof of Theorem \ref{thm:gPCcoefficients})}
	
	In this section, all the norms and inner products acting on $\phi_{k}$ are taken on the random space $\mathbb{Z}$ and with respect to the measure $\pi(z) \mathrm{d} z$. In order to prove the estimate for the gPC coefficients, we need the extra assumption \eqref{ineq:thmgPC1} on  basis functions.

	Before the proof, we state the following lemma approximating a nonlinear term, which is a key estimate analogous to Lemma 5.1 in \cite{ShuJin2018}:
	
	\begin{lemma}\label{Lem3.1}
		Assume $\left|S_{j l k}\right| \leq C j^{p}$, which follows from \eqref{ineq:thmgPC1}. Let $q>p+2$. Let $s>\frac{3}{2}$, $\alpha$ be a multi-index with $|\alpha| \leq s$. Let $u_{k}=$ $u_{k}(x) \in H^{s}, w_{k}=w_{k}(x) \in H^{s}, y_{k}=y_{k}(x) \in L^{2}, f_{k}=f_{k}(x, v) \in \tilde{H}^{2}, g_{k}=g_{k}(x, v) \in L^{2}$. Then
		$$
		\begin{aligned}
			&\left|\sum_{k=1}^{K} k^{2 q}\left\langle\partial^{\alpha}(u w)_{k}, y_{k}\right\rangle\right| \leq \dfrac{C}{\delta} \sum_{j=1}^{K}\left\|j^{q} u_{j}\right\|_{s}^{2} \sum_{l=1}^{K}\left\|l^{q} w_{l}\right\|_{s}^{2}+\delta \sum_{k=1}^{K}\left\|k^{q} y_{k}\right\|_{0}^{2}, \\
			&\left|\sum_{k=1}^{K} k^{2 q}\left\langle\partial^{\alpha}(u f)_{k}, g_{k}\right\rangle\right| \leq \dfrac{C}{\delta} \sum_{j=1}^{K}\left\|j^{q} u_{j}\right\|_{s}^{2} \sum_{l=1}^{K}\left\|l^{q} f_{l}\right\|_{s}^{2}+\delta \sum_{k=1}^{K}\left\|k^{q} g_{k}\right\|_{0}^{2},
		\end{aligned} 
		$$
		where the constants are independent of $K$, and $\delta$ is any positive constant.
	\end{lemma}
	
	\begin{remark}
		The weight $k^{q}$ appeared in the above lemma is essential. Suppose one uses a summation $\sum_{k=1}^{K}\left\langle\partial^{\alpha}(u w)_{k}, y_{k}\right\rangle$, then one ends up with the estimate
		$$
		\begin{aligned}
			\left|\sum_{k=1}^{K}\left\langle\partial^{\alpha}(u w)_{k}, y_{k}\right\rangle\right| &=\left|\sum_{j, l, k=1}^{K} S_{j l k}\left\langle\partial^{\alpha}\left(u_{j} w_{l}\right), y_{k}\right\rangle\right| \\
			& \leq \dfrac{C}{\delta} C_{1}(K) \sum_{j=1}^{K}\left\|u_{j}\right\|_{s}^{2} \sum_{l=1}^{K}\left\|w_{l}\right\|_{s}^{2}+\delta C_{2}(K) \sum_{k=1}^{K}\left\|y_{k}\right\|_{0}^{2},
		\end{aligned}
		$$
		where $C_{1}(K)=\sum_{k=1}^{K} k^{p}=O\left(K^{p+1}\right), \quad C_{2}(K)=K \sum_{j=1}^{K} j^{p}=O\left(K^{p+2}\right) .$ Thus in this way one gets an estimate with the coefficient depending on $K$. If this estimate is used to prove an analog of Theorem \ref{thm:gPCcoefficients}, then a $K$-dependent constant $c_{2}$ will be obtained.
		
		Given Proposition \ref{prop_gPC1}, $c_{2}$ being independent of $K$ means that the conclusion of Theorem \ref{thm:gPCcoefficients} holds if the initial data satisfies a $K$-independnet smoothness condition.
		If $c_{2}$ depends on $K$, then the initial data needs to satisfy a $K$-dependent condition to achieve the conclusion of Theorem \ref{thm:gPCcoefficients}. 
		This is not good because the gPC-sG method is expected to be stable for a class of initial data for all $K$.
	\end{remark}

	Due to the similarity of Lemma \ref{Lem1.3} and Lemma \ref{Lem3.1}, it is straightforward to modify the proof of Theorem \ref{thm:energyestimate} into a proof of Theorem \ref{thm:gPCcoefficients} :
	
	\begin{proof}
		The gPC coefficients of the mean fluid velocity satisfies
		\begin{equation}\label{pfgPC2_ubar}		
			\partial_{t} \bar{u}_{k}+\frac{\kappa}{\epsilon}\sum_{i=1}^{N}i^{1/3} \bar{u}_{k}+\dfrac{\kappa}{\epsilon}\frac{1}{\left|\mathbb{T}^{3}\right|}\sum_{i=1}^{N}i^{1/3} \int_{{\mathbb{T}}^{3}}\int_{{\mathbb{R}}^{3}} \sqrt{\mu_i}(u f_i)_{k} \mathrm{~d} v \mathrm{~d} x=\frac{\kappa}{\epsilon} \frac{1}{\left|\mathbb{T}^{3}\right|} \sum_{i=1}^{N} \int_{{\mathbb{T}}^{3}}\int_{{\mathbb{R}}^{3}}  i^{1/3} v \sqrt{\mu_i} f_{ik} \mathrm{d} v \mathrm{d} x ,
		\end{equation}
		and
		\begin{equation}\label{pfgPC2_ubar2}	-\frac{\kappa}{\epsilon} \frac{1}{\left|\mathbb{T}^{3}\right|} \sum_{i=1}^{N} \int_{{\mathbb{T}}^{3}}\int_{{\mathbb{R}}^{3}}  i v \sqrt{\mu_i} f_{ik} \mathrm{d} v \mathrm{d} x = \bar{u}_k.
		\end{equation}
		
		Take $\partial^{\alpha}$ on the first and second equations of \eqref{eq:ufk}, do $L^{2}$ estimates on them as well as on \eqref{pfgPC2_ubar} and \eqref{pfgPC2_ubar2}, and sum over $i$, $k$ and $\alpha$ with the $k$ th equation multiplied by $k^{2 q}$. Then one gets
		$$
		\frac{1}{2} \partial_{t} E^{K}+G^{K}+B^{K}\leq 0,
		$$
		where
		$$
		\begin{aligned}
			E^{K}=&\sum_{k=1}^{K}\left(\left\|k^{q} u_{k}\right\|_{s}^{2}+\kappa\bar{\theta}\sum_{i=1}^{N}\left\|k^{q} f_{ik}\right\|_{s}^{2}+\left|k^{q} \bar{u}_{k}\right|^{2}\right), \\
			G^{K}=&G_{1}^{K}+G_{2}^{K} = G_{1}^{K}+\dfrac{\kappa}{\epsilon}\sum_{i=1}^{N}i^{1/3}G_{2,i}^{K}, \\
			B^{K}=&B_{1}^{K}+B_{2}^{K}+B_{3}^{K}=\sum_{|\alpha| \leq s} B_{1, \alpha}^{K}+\sum_{|\alpha| \leq s}\left(\dfrac{\kappa}{\epsilon}\sum_{i=1}^{N}i^{1/3} B_{2,i, \alpha}^{K}\right)+\dfrac{\kappa}{\epsilon}\sum_{i=1}^{N}i^{1/3}B_{3,i}^{K},
		\end{aligned}
		$$	
		with
		$$
		\begin{aligned}
			G_1^K &= \sum_{k=1}^{K}\left(\left\|\nabla_{x} k^{q} u_{k}\right\|_{s}^{2}+\dfrac{\kappa}{\epsilon}\left(\sum_{i=1}^{N}i^{1/3}-1\right)\left|k^{q} \bar{u}_{k}\right|^{2}\right),\\
			G_{2,i}^K &= \sum_{k=1}^{K}\left\|k^{q}\left(u_{k} \sqrt{\mu_i}-\frac{\bar{\theta}}{i} \nabla_{v} f_{ik}- \frac{v}{2} f_{ik}\right)\right\|_{s}^{2},\\
			B_{1, \alpha}^{K} &=\sum_{k=1}^{K} k^{2 q}\left\langle\partial^{\alpha}\left(u \cdot \nabla_{x} u\right)_{k}, \partial^{\alpha} u_{k}\right\rangle, \\
			B_{2, i,\alpha}^{K} &=\sum_{k=1}^{K} k^{2 q}\left\langle\partial^{\alpha}(u f_i)_{k}, \partial^{\alpha}\left[u_{k} \sqrt{\mu_i}-\frac{\bar{\theta}}{i} \nabla_{v} f_{ik}-\frac{v}{2} f_{ik}\right]\right\rangle, \\
			B_{3,i}^{K} &=\frac{1}{\left|\mathbb{T}^{3}\right|} \sum_{k=1}^{K} k^{2 q}\left\langle(u f_i)_{k}, \bar{u}_{k} \sqrt{\mu_i}\right\rangle.
		\end{aligned}
		$$
		By applying Lemma \ref{Lem3.1}, one gets
		$$
		\begin{aligned}
			&\left|B_{1, \alpha}^{K}\right| \leq \dfrac{C}{\delta} \sum_{k=1}^{K}\left\|k^{q} u_{k}\right\|_{s+1}^{2} \sum_{k=1}^{K}\left\|k^{q} u_{k}\right\|_{s}^{2}+\delta \sum_{k=1}^{K}\left\|k^{q} u_{k}\right\|_{s+1}^{2} \leq \dfrac{C}{\delta} E^{K} G_{1}^{K}+\delta G_{1}^{K}, \\
			&\begin{aligned}
				\left|\dfrac{\kappa}{\epsilon}\sum_{i=1}^{N}i^{1/3}B_{2,i, \alpha}^{K}\right| &\leq \dfrac{C}{\delta} \sum_{k=1}^{K}\left(\sum_{i=1}^{N}\left\|k^{q} f_{ik}\right\|_{s}^{2}\right)\sum_{k=1}^{K}\left(\dfrac{\kappa}{\epsilon}\sum_{i=1}^{N}i^{1/3}\left\|k^{q} u_{k}\right\|_{s}^{2}\right) +\delta \dfrac{\kappa}{\epsilon}\sum_{i=1}^{N}i^{1/3} G_{2,i}^{K} \\
				&\leq \dfrac{C}{\delta} E^{K} G_{1}^{K}+\delta G_{2}^{K},
			\end{aligned}	 \\
			&\begin{aligned}
				\left|\dfrac{\kappa}{\epsilon}\sum_{i=1}^{N}i^{1/3} B_{3,i}^{K}\right| &\leq \dfrac{C}{\delta}\sum_{k=1}^{K}\sum_{i=1}^{N}\left\|k^{q} f_{ik}\right\|_{s}^{2} \sum_{k=1}^{K}\left(\dfrac{\kappa}{\epsilon}\sum_{i=1}^{N}i^{1/3}\left\|k^{q} u_{k}\right\|_{s}^{2}\right) + \delta\sum_{k=1}^{K}\left(\dfrac{\kappa}{\epsilon}\sum_{i=1}^{N}i^{1/3}\left|k^{q} \bar{u}_{k}\right|^{2}\right) \\
				&\leq \dfrac{C}{\delta} E^{K} G_{1}^{K}+\delta G_{1}^{K}.
			\end{aligned}
		\end{aligned}
		$$
		And then one concludes
		\begin{equation}\label{ineq:EG}
			\frac{1}{2} \partial_{t} E^{K} \leq-\left(1-\dfrac{C}{\delta} E^{K}-C \delta\right) G^{K}.
		\end{equation}
		Assume $\delta=\frac{1}{4 C}$ where $C$ is the constant in \eqref{ineq:EG}, and $c_{2}(s, r)=\frac{1}{16 C^2}$. Then by the same argument as in the proof of Theorem \ref{thm:energyestimate}, if $E^{K}(0) \leq c_{2}(s, r)$, one has
		$$
		\partial_{t} E^{K}+G^{K} \leq 0,
		$$
		and $E^{K}$ is non-increasing.
		\qed
	\end{proof}
	
	\subsection{Accuracy analysis (Proof of Theorem \ref{thm:gPCaccuracy})}
	
	We first state the following lemma, which is an estimate on $u^K$ from \cite{ShuJin2018}:
	
	\begin{lemma}
		Recall that the definition of the reconstructed gPC solution is $$
		u^{K}(x, z)=\sum_{k=1}^{K} u_{k}(x) \phi_{k}(z).
		$$
		Then at a fixed $x$ point one has
		$$
		\left\|u^{K}\right\|_{0}^{2} \leq E_{0, q}^{K},
		$$
		for any $q \geq 0$.
		Furthermore, with the assumption \eqref{ineq:thmgPC1},
		one has 
		$$
		\left\|u^{K}(x)\right\|_{L_{z}^{\infty}}^{2} \leq C\left(\sum_{k=1}^{K}\left|u_{k}(x)\right| k^{p}\right)^{2} \leq C\left(\sum_{k=1}^{K}\left|k^{q} u_{k}(x)\right|^{2}\right)\left(\sum_{k=1}^{K} k^{2(p-q)}\right) \leq C\left(\sum_{k=1}^{K}\left|k^{q} u_{k}(x)\right|^{2}\right),
		$$
		since $q>p+2$. Thus
		\begin{equation}\label{ineq:uK}
			\left\|u^{K}\right\|_{L_{z}^{\infty}\left(L_{x}^{2}\right)}^{2} \leq\left\|u^{K}\right\|_{L_{x}^{2}\left(L_{z}^{\infty}\right)}^{2} \leq C E_{0, q}^{K} .
		\end{equation}
		Similar estimates hold for $f_i$ and $\bar{u}$ and their $x$ derivatives.
	\end{lemma}
	
	Now give the proof of Theorem \ref{thm:gPCaccuracy}.
	
	\begin{proof}
		Denote the projection operator onto the span of $\left\{\phi_{k}\right\}_{k=1}^{K}$ by $P_{K}$. Multiplying \eqref{eq:ufk} and \eqref{pfgPC2_ubar}-\eqref{pfgPC2_ubar2} by $\phi_{k}(z)$ and summing in $k$, one gets the equations for $\left(u^{K}, f_i^{K}\right)$
		$$
		\begin{aligned}
			&\begin{aligned}\partial_{t} f_i^{K}+ v \cdot \nabla_{x} f_i^{K}+\frac{i^{1/3}}{\bar{\theta}\epsilon}\left(\frac{\bar{\theta}}{i}\nabla_{v}-\frac{v}{2}\right) \cdot P_{K}\left(u^{K} f_i^{K}\right)-\frac{i^{1/3}}{\bar{\theta}\epsilon} &u^{K} \cdot  v \sqrt{\mu_i}\\
				&=\frac{i^{1/3}}{\bar{\theta}\epsilon}\left(\frac{-|v|^{2}}{4}+\frac{3\bar{\theta}}{2i}+\frac{\bar{\theta}^2}{i^2}\Delta_{v}\right) f_i^{K},\end{aligned} \\
			&\begin{aligned}
				\partial_{t} u^{K}+P_{K}\left(u^{K} \cdot \nabla_{x} u^{K}\right)+\nabla_{x} p^{K}-\Delta_{x} u^{K}+\frac{\kappa}{\epsilon}\sum_{i=1}^{N}i^{1/3}u^{K}+\frac{\kappa}{\epsilon}\sum_{i=1}^{N}i^{1/3}\int \sqrt{\mu_i} P_{K}\left(u^{K} f_i^{K}\right) \mathrm{d} v\\
				-\frac{\kappa}{\epsilon}\sum_{i=1}^{N}i^{1/3} \int v \sqrt{\mu_i} f_i^{K} \mathrm{~d} v=0,
			\end{aligned} \\
			&\nabla_{x} \cdot u^{K}=0, \\
			&\partial_{t} \bar{u}^{K}+\frac{\kappa}{\epsilon}\sum_{i=1}^{N}i^{1/3} \bar{u}^{K}+\dfrac{\kappa}{\epsilon}\frac{1}{\left|\mathbb{T}^{3}\right|}\sum_{i=1}^{N}i^{1/3} \int_{{\mathbb{T}}^{3}}\int_{{\mathbb{R}}^{3}} \sqrt{\mu_i} P_{K}\left(u^{K} f_i^{K}\right) \mathrm{d} v \mathrm{~d} x\\
			&\hspace{25em}=\frac{\kappa}{\epsilon} \frac{1}{\left|\mathbb{T}^{3}\right|} \sum_{i=1}^{N} \int_{{\mathbb{T}}^{3}}\int_{{\mathbb{R}}^{3}}  i^{1/3} v \sqrt{\mu_i} f_{i}^K \mathrm{d} v \mathrm{d} x ,\\
			&-\frac{1}{|\mathbb{T}|^{3}}\sum_{i=1}^{N}\int_{{\mathbb{T}}^{3}}\int_{{\mathbb{R}}^{3}}i v \sqrt{\mu_i} f_i^K \mathrm{d} v \mathrm{d} x=\bar{u}^K.
		\end{aligned}
		$$
		Subtracting from \eqref{eq:uf} and \eqref{eq:ubar}-\eqref{eq:ubar2}, one gets
		\begin{equation}\label{eq:ufe}
			\begin{aligned}
				&\begin{aligned}
					\partial_{t} f_i^{e}+ v \cdot \nabla_{x} f_i^{e}+\frac{i^{1/3}}{\bar{\theta}\epsilon}\left(\frac{\bar{\theta}}{i}\nabla_{v}-\frac{v}{2}\right) \cdot\left[\left(I-P_{K}\right)(u f_i)+P_{K}\left(u^{e} f_i+u^{K} f_i^{e}\right)\right]-\frac{i^{1/3}}{\bar{\theta}\epsilon} u^{e} \cdot v \sqrt{\mu_i} \\
					=\frac{i^{1/3}}{\bar{\theta}\epsilon}\left(\frac{-|v|^{2}}{4}+\frac{3\bar{\theta}}{2i}+\frac{\bar{\theta}^2}{i^2}\Delta_{v}\right) f_i^{e}, 
				\end{aligned}\\
				&\partial_{t} u^{e}+\left[\left(I-P_{K}\right)\left(u \cdot \nabla_{x} u\right)+P_{K}\left(u^{e} \cdot \nabla_{x} u+u^{K} \cdot \nabla_{x} u^{e}\right)\right]+\nabla_{x} p^{e}-\Delta_{x} u^{e}+\dfrac{\kappa}{\epsilon}\sum_{i=1}^{N}i^{1/3}u^{e} \\
				&\  +\dfrac{\kappa}{\epsilon}\sum_{i=1}^{N}i^{1/3}\int \sqrt{\mu_i}\left[\left(I-P_{K}\right)(u f_i)+P_{K}\left(u^{e} f_i+u^{K} f_i^{e}\right)\right] \mathrm{d} v-\dfrac{\kappa}{\epsilon}\sum_{i=1}^{N}i^{1/3} \int v \sqrt{\mu_i} f_i^{e} \mathrm{~d} v=0,\\
				&\nabla_{x} \cdot u^{e}=0, \\
				&\begin{aligned}
					\partial_{t} \bar{u}^{e}+\frac{\kappa}{\epsilon}\sum_{i=1}^{N}i^{1/3} \bar{u}^{e}+\dfrac{\kappa}{\epsilon}\frac{1}{|\mathbb{T}|^{3}}\sum_{i=1}^{N}i^{1/3} \int_{{\mathbb{T}}^{3}}\int_{{\mathbb{R}}^{3}}  \sqrt{\mu_i}\left[\left(I-P_{K}\right)(u f_i)+P_{K}\left(u^{e} f_i+u^{K} f_i^{e}\right)\right] \mathrm{d} v \mathrm{~d} x\\
					= \frac{\kappa}{\epsilon} \frac{1}{|\mathbb{T}|^{3}} \sum_{i=1}^{N} \int_{{\mathbb{T}}^{3}}\int_{{\mathbb{R}}^{3}}  i^{1/3} v \sqrt{\mu_i} f_{i}^e \mathrm{d} v \mathrm{d} x ,
				\end{aligned}\\
				&-\frac{1}{|\mathbb{T}|^{3}}\sum_{i=1}^{N}\int_{{\mathbb{T}}^{3}}\int_{{\mathbb{R}}^{3}}i v \sqrt{\mu_i} f_i^e \mathrm{d} v \mathrm{d} x=\bar{u}^e,
			\end{aligned}
		\end{equation}
		where $\left(u^{e}, \{f_i^{e}\}_{i=1}^N\right)$ is the approximation error
		$$
		u^{e}=u-u^{K}, \quad f_i^{e}=f_i-f_i^{K},\  i=1,2,\ldots,N .
		$$
		Note that \eqref{eq:ufe} is linear in $\left(u^{e}, \{f_i^{e}\}_{i=1}^N\right)$.
		
		Take $\partial^{\alpha}$ on the first and second equations of \eqref{eq:ufe}, and do $L^{2}$ estimates on the first, second, fourth and fifth equations in $(x, z)$, $(x, v, z)$, $z$ and  $z$, respectively.
		Note that $P_{K}$ commutes with $x$-derivatives and has operator norm 1 on $L_{z}^{2}$. Thus one has
		$$
		\left|\left\langle\left\langle\partial^{\alpha} P_{K}\left(u^{e} \cdot \nabla_{x} u+u^{K} \cdot \nabla_{x} u^{e}\right), \partial^{\alpha} u^{e}\right\rangle\right\rangle\right| \leq C\left(\|u\|_{W^{s+1, \infty}}+\left\|u^{K}\right\|_{W^{s, \infty}}\right)\left\|u^{e}\right\|_{s+1}^{2},
		$$
		where the $W$ norm is defined in \eqref{def:finfty} and the sub-index $r=0$ is omitted. By estimating the terms $P_{K}\left(u^{e} f_i+u^{K} f_i^{e}\right)$ in the same manner, one has
		$$
		\left|\left\langle\left\langle\partial^{\alpha} P_{K}\left(u^{e} f_i+u^{K} f_i^{e}\right), \partial^{\alpha} u^{e}\right\rangle\right\rangle\right| \leq C\left(\|f_i\|_{W^{s, \infty}}+\left\|u^{K}\right\|_{W^{s, \infty}}\right)\left(\left\|u^{e}\right\|_{s}^{2}+\left\|f_i^{e}\right\|_{s}^{2}\right).
		$$
		Adding the above together, one gets the energy estimate
		\begin{equation}\label{ineq:EGS}
			\frac{1}{2} \partial_{t} E^{e}+\frac{2}{3} G^{e} \leq C H G^{e}+C S,
		\end{equation}
		where
		$$
		\begin{aligned}
			&E^{e}=\left\|u^{e}\right\|_{s}^{2}+\kappa\bar{\theta}\sum_{i=1}^N \left\|f_i^{e}\right\|_{s}^{2}+\left\|\bar{u}^{e}\right\|^{2}, \\
			&G^{e}=\left\|\nabla_{x} u^{e}\right\|_{s}^{2}+\dfrac{\kappa}{\epsilon}\left(\sum_{i=1}^Ni^{1/3}-1\right)\left\|\bar{u}^{e}\right\|^{2}+\dfrac{\kappa}{\epsilon}\sum_{i=1}^Ni^{1/3}\left\|u^{e} \sqrt{\mu_i}-\frac{\bar{\theta}}{i} \nabla_{v} f_i^{e}- \frac{v}{2} f_i^{e}\right\|_{s}^{2}, \\
			&S=\left\|\left(I-P_{K}\right)\left(u \cdot \nabla_{x} u\right)\right\|_{s}^{2}+\sum_{i=1}^N\left\|\left(I-P_{K}\right)(u f_i)\right\|_{s}^{2}, \\
			&H=\|u\|_{W^{s+1, \infty}}+\left\|u^{K}\right\|_{W^{s, \infty}}+\sum_{i=1}^N \|f_i\|_{W^{s, \infty}}.
		\end{aligned}
		$$
		
		Notice that by Sobolev embedding,
		$$
		\|u\|_{W^{s+1, \infty}} \leq C\|u\|_{L_{z}^{\infty}\left(H_{x}^{s+3}\right)}, \quad\|f_i\|_{W^{s, \infty}} \leq C\|f_i\|_{L_{z}^{\infty}\left(H_{x}^{s+2}\left(L_{v}^{2}\right)\right)},
		$$
		and by \eqref{ineq:uK}
		$$
		\left\|u^{K}\right\|_{W^{s, \infty}}^{2} \leq C E_{s+2, q}^{K}.
		$$
		Thus $H$ can be controlled by
		$$
		H \leq C\left(\left\|E_{s+3,0}\right\|_{L_{z}^{\infty}}+E_{s+2, q}^{K}\right)^{1 / 2} .
		$$
		In view of Proposition \ref{prop_gPC1}, for $r>p+\frac{5}{2}$, one has
		$$
		H \leq C\left\|E_{s+3, r}\right\|_{L_{z}^{\infty}}^{1 / 2},
		$$
		which implies that
		\begin{equation}\label{ineq:CH}
			C H \leq \frac{1}{6},
		\end{equation}
		in \eqref{ineq:EGS} for all time if $\left\|E_{s+3, r}(0)\right\|_{L_{z}^{\infty}} \leq c_{1}^{\prime \prime}(s, r) \leq \min \left\{\frac{1}{4 C}, c_{1}(s, r), c_{2}(s, q)\right\}$, in view of Theorem \ref{thm:energyestimate} and Theorem \ref{thm:gPCcoefficients}.
		
		To estimate the source term $S$, notice that at each fixed $x, v$,
		$$
		\left\|\left(I-P_{K}\right) \partial^{\alpha}(u f_i)(x, v)\right\|_{L_{z}^{2}} \leq C \frac{\left\|\partial^{\alpha}(u f_i)(x, v)\right\|_{H_{z}^{r}}}{K^{r}} .
		$$
		Integrating in $x, v$ and summing over $\alpha$,
		$$
		\left\|\left(I-P_{K}\right)(u f_i)\right\|_{s} \leq C \frac{\|u f_i\|_{s, r}}{K^{r}} .
		$$
		
		Notice that at each $z$,
		$$
		|u f_i|_{s, r} \leq \max _{|\alpha| \leq s,|\gamma| \leq r}\left\|\partial^{\alpha} u^{\gamma}\right\|_{L_{x}^{\infty}}|f_i|_{s, r} \leq C|u|_{s+2, r}|f_i|_{s, r} .
		$$
		Thus
		$$
		\sum_{i=1}^N\|u f_i\|_{s, r} \leq\sum_{i=1}^N\left\||u f_i|_{s, r}\right\|_{L_{z}^{\infty}} \leq C\left\| |u|_{s+2, r}\right\|_{L_{z}^{\infty}}\sum_{i=1}^N\left\||f_i|_{s, r}\right\|_{L_{z}^{\infty}} \leq C\left\|E_{s+2, r}\right\|_{L_{z}^{\infty}}^{1 / 2}\left\|E_{s, r}\right\|_{L_{z}^{\infty}}^{1 / 2} .
		$$
		Then by Theorems \ref{thm:energyestimate} and \ref{thm:hypocoercivity} (suppress the dependence on $C^{h}$ ), taking $c_{1}^{\prime \prime} \leq c_{1}^{\prime}(s, r)$,
		$$
		E_{s+2, r}(t) \leq C, \quad E_{s, r}(t) \leq C e^{-\lambda t} .
		$$
		Thus one finally gets
		$$
		\sum_{i=1}^N\left\|\left(I-P_{K}\right)(u f_i)\right\|_{s} \leq \frac{C e^{-\frac{\lambda}{2} t}}{K^{r}}.
		$$
		The term $\left\|\left(I-P_{K}\right)\left(u \cdot \nabla_{x} u\right)\right\|_{s}$ can be estimated similarly, by using $\left|u \cdot \nabla_{x} u\right|_{s, r} \leq C|u|_{s+3, r}|u|_{s, r}$, and one gets
		\begin{equation}\label{ineq:SK}
			S \leq \frac{C e^{-\lambda t}}{K^{2 r}}.
		\end{equation}
		
		In conclusion, combining \eqref{ineq:EGS}, \eqref{ineq:CH} and \eqref{ineq:SK}, one has the estimate
		\begin{equation}\label{ineq:EeGeK}
			\partial_{t} E^{e}+G^{e} \leq \frac{C}{K^{2 r}} e^{-\lambda t}.
		\end{equation}
		Noticing that $\int_{0}^{\infty} e^{-2 \lambda t} \mathrm{~d} t$ converges, one concludes that $E^{e} \leq \frac{C}{K^{2 r}}$ uniformly in time and $\epsilon$.
		\qed
	\end{proof}
	
	\subsection{Hypocoercivity estimates for the error (Proof of Theorem \ref{thm:gPChypocoercivity})}
	
	Before the proof, first state the following lemma, which is Lemma 5.4 in \cite{ShuJin2018}.
	
	\begin{lemma}\label{Lem4}
		Let $\Phi=\Phi(t)$ satisfy
		$$
		\frac{\mathrm{d} \Phi}{\mathrm{d} t}+a_{1} \Phi \leq a_{2} e^{-a_{3} t}.
		$$
		Then
		$$
		\Phi(t) \leq e^{-a t}\left(\Phi(0)+a_{2} C(\delta)\right),
		$$
		with $a=\min \left\{a_{1}, a_{3}\right\}-\delta$, $\delta$ being any positive constant.
	\end{lemma}
	
	Now it is ready to prove Theorem \ref{thm:gPChypocoercivity}.
	
	\begin{proof}
		In order to get a hypocoercivity estimate for $\left(u^{e}, \{f_i^{e}\}_{i=1}^N\right)$, we write the equation of $\partial^{\alpha} f_i^{e}$ as
		\begin{equation}\label{eq:fe}	
			\begin{aligned}
				\partial_{t} \partial^{\alpha} f_i^e+&\dfrac{i}{\bar{\theta}}\mathcal{P}_i \partial^{\alpha} f_i^e+\frac{i^{1/3}}{\bar{\theta}\epsilon}\left(\mathcal{K}_i^{*} \cdot \mathcal{K}_i\right) \partial^{\alpha} f_i^e= \frac{i^{1/3}}{\bar{\theta}\epsilon} \partial^{\alpha} u^{e} \cdot v \sqrt{\mu_i} \\
				&+\frac{i^{1/3}}{\bar{\theta}\epsilon} \left[\left(I-P_{K}\right) \partial^{\alpha}\left(u \cdot \mathcal{K}_i^{*} f_i\right)+P_{K} \partial^{\alpha}\left(u^{e} \cdot \mathcal{K}_i^{*} f_i\right) +P_{K} \partial^{\alpha}\left(u^{K} \cdot \mathcal{K}_i^{*} f_i^{e}\right)\right].
			\end{aligned}
		\end{equation}
		Then do energy estimate in $(x, v, z)$. 
		
		The linear terms can be handled in the same way as Theorem \ref{thm:hypocoercivity}. The nonlinear terms can be estimated as follows:
		$$
		\begin{aligned}
			\left|\frac{i^{1/3}}{\epsilon}\left(\left(\left(I-P_{K}\right) \partial^{\alpha}\left(u \cdot \mathcal{K}_i^{*} f_i\right), \partial^{\alpha} f_i^{e}\right)\right)\right|
			=&\left| \dfrac{2}{\epsilon}\left\langle\left\langle\left(I-P_{K}\right) \partial^{\alpha}(u \mathcal{K}_i f_i), \mathcal{K}_i^{2} \partial^{\alpha} f_i^{e}\right\rangle\right\rangle  + \text { similar terms }\right|\\
			\leq & \frac{C}{K^{r}} \frac{1}{\epsilon^{2}}\|u\|_{L_{z}^{\infty}\left(H^{s}+3, r\right)}\left([[f_i, f_i]]_{s, r}+\left[\left[f_i^{e}, f_i^{e}\right]\right]_{s}\right),
		\end{aligned}
		$$
		\begin{equation}\label{ineq:nonlinearfe}
			\begin{aligned}
				\left|\frac{i^{1/3}}{\epsilon}\left(\left(P_{K} \partial^{\alpha}\left(u^{e} \cdot \mathcal{K}_i^{*} f_i\right), \partial^{\alpha} f_i^{e}\right)\right)\right|  \leq&\left|\frac{i^{1/3}}{\epsilon}\left(\left(\partial^{\alpha}\left(u^{e} \cdot \mathcal{K}_i^{*} f_i\right), \partial^{\alpha} f_i^{e}\right)\right)\right| \\
				\leq & C \frac{1}{\epsilon^{2}} \max _{|\beta| \leq s}\left\|\partial^{\beta} u^{e}\right\|_{L^{\infty}}\left(\dfrac{C}{\delta}[[f_i, f_i]]_{s}+\delta\left[\left[f_i^{e}, f_i^{e}\right]\right]_{s}\right),
			\end{aligned}
		\end{equation}
		and
		$$
		\begin{aligned}			\left|\frac{i^{1/3}}{\epsilon}\left(\left(P_{K} \partial^{\alpha}\left(u^{K} \cdot \mathcal{K}_i^{*} f_i^{e}\right), \partial^{\alpha} f_i^{e}\right)\right)\right| \leq&\left|\frac{i^{1/3}}{\epsilon}\left(\left(\partial^{\alpha}\left(u^{K} \cdot \mathcal{K}_i^{*} f_i^{e}\right), \partial^{\alpha} f_i^{e}\right)\right)\right| \leq C \frac{1}{\epsilon^{2}}\left\|u^{K}\right\|_{L_{z}^{\infty}\left(H^{s+3}\right)}\left[\left[f_i^{e}, f_i^{e}\right]\right]_{s}.
		\end{aligned}$$
		Now by assumption, $\left\|E_{s+3, r}(t)\right\|_{L_{z}^{\infty}}$ is small enough at $t=0$ (which implies that they are small enough for all time by Theorem \ref{thm:energyestimate}). Similar result holds for $E_{s+3, q}^{K} \leq C\left\|E_{s+3, r}\right\|_{L_{z}^{\infty}}$ by Theorem \ref{thm:gPCcoefficients}. As a result, $\|u\|_{L_{z}^{\infty}\left(H^{s+3, r}\right)}$ and $\left\|u^{K}\right\|_{L_{z}^{\infty}\left(H^{s+3}\right)}$ are small enough, see \eqref{ineq:uK} for the latter.
		
		To bound the term $\max _{|\beta| \leq s}\left\|\partial^{\beta} u^{e}\right\|_{L^{\infty}}$ appeared in \eqref{ineq:nonlinearfe}, one estimates
		$$
		\left\|u^{e}\right\|_{L_{z}^{\infty}}=\left\|\sum_{k=1}^{K}\left(u^{e}\right)_{k} \phi_{k}(z)\right\|_{L_{z}^{\infty}} \leq C\left(\sum_{k=1}^{K}\left|\left(u^{e}\right)_{k}\right|^{2}\right)^{1 / 2}\left(\sum_{k=1}^{K} k^{2 p}\right)^{1 / 2} \leq C\left\|u^{e}\right\|_{L_{z}^{2}} K^{p+1 / 2},
		$$
		at any fixed $x$. By taking $L^{\infty}$ in $x$, one obtains $$ \left\|u^{e}\right\|_{L^{\infty}} \leq C K^{p+1 / 2}\left\|u^{e}\right\|_{L_{x}^{\infty}\left(L_{z}^{2}\right)} \leq C K^{p+1 / 2}\left\|u^{e}\right\|_{L_{z}^{2}\left(L_{x}^{\infty}\right)} \leq C K^{p+1 / 2}\left\|u^{e}\right\|_{L_{z}^{2}\left(H_{x}^{2}\right)}, $$
		and
		$$
		\max _{|\beta| \leq s}\left\|\partial^{\beta} u^{e}\right\|_{L^{\infty}} \leq \frac{C}{K^{r-p-1 / 2}} .
		$$
		
		Then by choosing $\delta$ in \eqref{ineq:nonlinearfe} small enough, all the $\left[\left[f_i^{e}, f_i^{e}\right]\right]_{s}$ terms from the nonlinear terms can be absorbed by the corresponding term from the linear terms, and then one concludes the estimate
		\begin{equation}\label{ineq:fe}
			\partial_{t}\left(\left(f_i^{e}, f_i^{e}\right)\right)_{s}+\frac{\lambda_{1}^{e}}{\bar{\theta}\epsilon^{2}}\left[\left[f_i^{e}, f_i^{e}\right]\right]_{s} \leq \frac{C\left(\lambda_{1}^{e}\right)}{\bar{\theta}}\left(\left\|u^{e}\right\|_{s}^{2}+\left\|\nabla_{x} u^{e}\right\|_{s}^{2}+\frac{1}{\epsilon^{2}}\left\|\mathcal{K}_i f_i^{e}\right\|_{s}^{2}\right)+\frac{C}{\bar{\theta}K^{r-p-1 / 2}} \frac{1}{\epsilon^{2}}[[f_i, f_i]]_{s, r}.		
		\end{equation}
		
		Finally, multiplying by $\kappa \bar{\theta}$ and summing over $i$, similar to the proof of Theorem \ref{thm:hypocoercivity}, by taking a suitable linear combination of \eqref{ineq:fe}, \eqref{ineq:EeGeK} and \eqref{ineq:Etilde0} integrated in $z$ (where the appearance of \eqref{ineq:Etilde0} is to control the term $[[f_i, f_i]]_{s, r}$ in \eqref{ineq:fe}), one gets
		$$	
		\partial_{t} \tilde{E}^{e}+ \tilde{G}^{e} \leq \lambda_{4}^{e}\tilde{B}^{e} + \lambda_{4}^{e}\kappa \frac{C}{K^{r-p-1 / 2}} \frac{1}{\epsilon^{2}}\sum_{i=1}^N[[f_i, f_i]]_{s, r}+\frac{C}{K^{r}} e^{-\lambda t}	,	
		$$
		where
		$$
		\tilde{E}^{e}=E^{e}+\lambda_{4}^{e}\kappa\bar{\theta}\sum_{i=1}^N\left(\left(f_i^{e}, f_i^{e}\right)\right)_{s}+\frac{1}{K^{r-p-1 / 2}} \lambda_{5}^{e}\|\tilde{E}\|_{L_{z}^{1}},
		$$
		$$
		\tilde{G}^{e}=G^{e}+\lambda_{4}^{e} \lambda_{1}^{e}\kappa\frac{1}{\epsilon^{2}}\sum_{i=1}^N \left[\left[f_i^{e}, f_i^{e}\right]\right]_{s}+\frac{1}{2 K^{r-p-1 / 2}} \lambda_{5}^{e}\|\tilde{G}\|_{L_{z}^{1}},
		$$
		and
		$$
		\tilde{B}^{e} = C\left(\lambda_{1}^{e}\right)\kappa\sum_{i=1}^N\left(\left\|u^{e}\right\|_{s}^{2}+\left\|\nabla_{x} u^{e}\right\|_{s}^{2}+\frac{1}{\epsilon^{2}}\left\|\mathcal{K}_i f_i^{e}\right\|_{s}^{2}\right).
		$$
		Choose $\lambda_{4}^{e}$ in the same way as the choice of $\lambda_{4}$ and one gets
		\begin{equation}\label{ineq:Ee}		
			\partial_{t} \tilde{E}^{e}+\frac{1}{2} \tilde{G}^{e} \leq \lambda_{4}^{e}\kappa \frac{C}{K^{r-p-1 / 2}} \frac{1}{\epsilon^{2}}\sum_{i=1}^N[[f_i, f_i]]_{s, r}+\frac{C}{K^{r}} e^{-\lambda t}	.	
		\end{equation}
		To choose $\lambda_{5}^{e}$, one wants the $\tilde{G}$ term to control the first RHS term in \eqref{ineq:Ee}, and thus choose
		$$
		\lambda_{5}^{e}= \frac{4 C \lambda_{4}^{e}}{\lambda_{4} \lambda_{1}},
		$$
		where the $C$ is the first constant in \eqref{ineq:Ee}. Then
		$$
		\partial_{t} \tilde{E}^{e}+\frac{1}{4} \tilde{G}^{e} \leq \frac{C}{K^{r}} e^{-\lambda t}.
		$$
		
		Then since $\tilde{E}^{e} \leq C \tilde{G}^{e}$ (which can be proved similarly as the proof of $\tilde{E} \leq C \tilde{G}$, see \eqref{ineq:EtildeG} ), and $\tilde{E}^{e}(0) \leq \frac{C}{K^{r-p-1 / 2}}$, one can conclude from Lemma \ref{Lem4} that
		$$
		\tilde{E}^{e} \leq \frac{C}{K^{r-p-1 / 2}} e^{-\lambda^{e} t},
		$$
		where $\lambda^{e}=\min \left\{\lambda, \frac{1}{4 C}\right\}-\delta$ for some $\delta>0$ small enough.
		\qed
	\end{proof}

	\section{Conclusion}
	For a kinetic-fluid model with random initial inputs which describes a mixture of dispersed particles of different sizes interacting with a fluid flow, in the fine particle regime, 
	under the assumption that  the random perturbation on initial fluid velocity as well as initial particle distribution is small in suitable Sobolev spaces with vanishing total mass and momentum, we proved uniform regularity of the solution  by energy estimates and the energy decays exponentially in time using hypocoercivity arguments. These results imply  that the long-term behavior of the solution is insensitive to random perturbation of the initial data.
	For random initial data near the global equilibrium, the generalized polynomial chaos expansion based stochastic Gelerkin  method is proved to have spectral accuracy, uniformly in time and the Knudsen number, and  captures the long-time behavior of the solution with an error that decays exponentially in time.

	\section*{Acknowledgement}
	S. Jin was partially supported by National Key R\&D Program of China (no. 21Z010300242) and National Natural Science Foundation of China (no. 20Z103020029). Y. Lin was partially supported by China Postdoctoral Science Foundation (no. 2021M702142 and no. 2021TQ0203). Y. Lin thanks Dr.  Ruiwen Shu for his  help during the preparation of the paper.

	\bibliography{mybibfile}
	
\end{document}